\documentclass[10pt, a4paper]{article}

\usepackage[utf8]{inputenc}
\usepackage[T1]{fontenc}
\usepackage[english]{babel}
\usepackage[usenames,dvipsnames]{color}
\usepackage{euscript}
\usepackage{bbm}
\usepackage{amsmath}
\usepackage{amsfonts}
\usepackage{amssymb}
\usepackage{float}
\usepackage{graphicx}
\usepackage{hyperref}
\usepackage{xcolor}
\hypersetup{linkcolor=blue}
\hypersetup{colorlinks=true}
\usepackage[margin=2.5cm]{geometry}
\usepackage{pstricks}

\newcommand{\fm}[1]{{#1}}

\usepackage{amsthm}
\theoremstyle{plain}
\newtheorem{theorem}{Theorem}
\newtheorem{lemma}{Lemma}
\newtheorem{remark}{Remark}

\newtheorem{definition}{Definition}

\newcommand{\N}{\mathbb{N}}
\newcommand{\R}{\mathbb{R}}
\newcommand{\dif}{\mathrm{d}}
\newcommand{\dt}{\dif t}

\newcommand{\dx}{\dif x}
\newcommand{\dz}{\dif z}
\newcommand{\ddt}{\frac{\dif}{\dt}}
\newcommand{\diverg}{\mathop{\mathrm{div}}}
\newcommand{\supp}{\mathop{\text{supp}}}
\newcommand{\rot}{\nabla \times}
\newcommand{\tanpart}[1]{\left[ #1 \right]_{\textrm{tan}}}

\newcommand{\ue}{u^\varepsilon}
\newcommand{\uinit}{{u_{*}}}
\newcommand{\Rem}{{r^\varepsilon}}
\newcommand{\Omext}{\mathcal{O}}
\newcommand{\BorderExt}{\partial\mathcal{O}}
\newcommand{\BorderNav}{\partial\Omega \setminus \Gamma}
\newcommand{\use}{u^{(\varepsilon)}}
\newcommand{\eval}[1]{\left\{#1\right\}}
\newcommand{\Hspace}{L^2_\gamma(\Omega)}
\newcommand{\ldiv}{L^2_{\mathrm{div}}(\Omext)}
\newcommand{\MW}{M_\mathrm{w}}
\newcommand{\fourierz}{\zeta}
\newcommand{\force}{\xi}
\newcommand{\controlspaceforce}{H^1((0,T), L^2(\Omext)) \cap 
\mathcal{C}^0([0,T],H^1(\Omext))}

\title{\fm{Small-time} global exact controllability of the Navier-Stokes 
equation with Navier slip-with-friction boundary conditions\thanks{Work 
supported by ERC Advanced Grant 266907 (CPDENL) of 
the 7th Research Framework Programme (FP7).}}

\author{Jean-Michel Coron\thanks{Sorbonne Universit\'es, UPMC Univ Paris 
06, Laboratoire Jacques-Louis Lions, UMR CNRS 7598.},
Frédéric Marbach\footnotemark[2]~~and
Franck Sueur\thanks{Institut de Mathématiques de Bordeaux, 
UMR CNRS 5251, Université de Bordeaux.}}

\begin{document}

\maketitle

\begin{abstract}
In this work, we investigate the \fm{small-time} global exact controllability 
\fm{of} the Navier-Stokes equation, both towards the null equilibrium state and 
towards weak trajectories. We consider a viscous incompressible fluid evolving 
within a smooth bounded domain, either in 2D or in 3D. The controls are only 
located on a small part of the boundary, intersecting all its connected 
components. On the remaining parts of the boundary, the fluid obeys a Navier 
slip-with-friction boundary condition. Even though viscous boundary layers 
appear near these uncontrolled boundaries, we prove that \fm{small-time} global 
exact controllability holds. Our analysis relies on the controllability of the 
Euler equation combined with asymptotic boundary layer expansions. Choosing the 
boundary controls with care enables us to guarantee good dissipation properties 
for the residual boundary layers, which can then be exactly canceled using 
local techniques.
\end{abstract}

% ==============================================================================
\section{Introduction}
% ==============================================================================

% ==============================================================================
\subsection{Description of the fluid system}

We consider a smooth bounded connected domain~$\Omega$ in $\R^d$, with 
$d = 2$ or $d = 3$.
Although some drawings will depict~$\Omega$ as a very simple domain, 
we do not make any other topological assumption on~$\Omega$. Inside this 
domain, an incompressible viscous fluid evolves under the Navier-Stokes 
equations. We will name~$u$ its velocity field and $p$ the associated pressure. 
We assume that we are able to act on the fluid flow only on a open part 
$\Gamma$ of the full boundary~$\partial\Omega$, where $\Gamma$ 
intersects all connected components of $\partial\Omega$ (this geometrical 
hypothesis is used in the proofs of Lemma~\ref{lemma.euler}). On the remaining 
part of the boundary, $\BorderNav$, we assume that 
the fluid flow satisfies Navier slip-with-friction boundary conditions.
Hence, $(u,p)$ satisfies:
\begin{equation} \label{eq.main}
    \left\{
    \begin{aligned}
        \partial_t u + (u \cdot \nabla) u - \Delta u + \nabla p & = 0
        && \quad \textrm{in } \Omega, \\
        \diverg u & = 0
        && \quad \textrm{in } \Omega, \\
        u \cdot n & = 0
        && \quad \textrm{on } \BorderNav, \\
        N(u) & = 0
        && \quad \textrm{on } \BorderNav.
    \end{aligned}
    \right.
\end{equation}
Here and in the sequel, $n$ denotes the outward pointing normal to the domain. 
For a vector field $f$, we introduce~$\tanpart{f}$ its tangential part, 
$D(f)$ the rate of strain tensor (or shear stress) and $N(f)$ the tangential 
Navier boundary operator defined as:
\begin{align} 
  \label{eq.def.tanpart}
  \tanpart{f} &:= f - (f \cdot n) n, \\
  \label{eq.def.shear}
  D_{ij}(f) &:= \frac{1}{2}\left(\partial_i f_j + \partial_j f_i\right), \\
  \label{eq.def.navier}
  N(f) &:= \tanpart{D(f)n + Mf}. 
\end{align}
Eventually, in~\eqref{eq.def.navier}, $M$ is a smooth matrix valued function, 
describing the friction near the boundary. This is a generalization of the 
usual condition involving a single scalar parameter $\alpha \geq 0$ (i.e. 
$M = \alpha I_d$). For flat boundaries, such a scalar coefficient measures the 
amount of friction. When $\alpha = 0$ and the boundary is flat, the fluid slips 
along the boundary without friction. When 
$\alpha \rightarrow +\infty$, the friction is so intense that the fluid is 
almost at rest near the boundary \fm{and, as shown by Kelliher 
in~\cite{MR2217315}, the Navier condition 
$\tanpart{D(u)n + \alpha u} = 0$ converges to the usual Dirichlet condition.}

% ==============================================================================
\subsection{Controllability problem and main result}

Let $T$ be an allotted positive time (possibly very small) and $\uinit$ an 
initial data (possibly very large). The question of \fm{small-time} global 
exact 
null controllability asks whether, for any $T$ and any $\uinit$, there exists a 
trajectory $u$ (in some appropriate functional space) defined on $[0,T] \times 
\Omega$, which is a solution to~\eqref{eq.main}, satisfying $u(0, \cdot) = 
\uinit$ and $u(T,\cdot) = 0$. In this formulation, system~\eqref{eq.main} is 
seen as an \fm{underdetermined} system. The controls used are the implicit 
boundary 
conditions on~$\Gamma$ and can be recovered from the constructed trajectory 
\emph{a posteriori}.

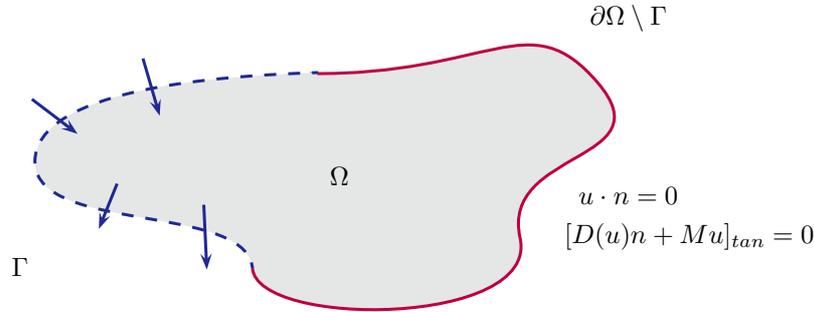
\begin{figure}[!ht]
 \centering
 \begin{pspicture}(-0.5,-2.236719)(11.0,2.2567186)
 \definecolor{gray}{rgb}{0.8980392156862745,0.9058823529411765,0.9058823529411765}
 \definecolor{blue}{rgb}{0.11764705882352941,0.1607843137254902,0.5725490196078431}
 \psbezier[linewidth=0.04,linecolor=purple,fillstyle=solid,fillcolor=gray](3.0554688,-1.4167187)(3.1154687,-2.2167187)(6.835469,-2.1967187)(6.5754685,-1.0167187)(6.315469,0.16328125)(8.4982395,-0.021577183)(7.595469,1.0432812)(6.6926975,2.1081395)(6.2,1.1367188)(3.8154688,1.1632812)
 \psbezier[linewidth=0.04,linecolor=blue,fillstyle=solid,fillcolor=gray,linestyle=dashed,dash=0.16cm
 0.16cm](3.0554688,-1.4167187)(2.9754689,-0.5367187)(0.14229742,-0.9753041)(0.19546875,0.0232812)(0.2486401,1.0218666)(2.9354687,1.1432812)(3.92,1.1767187)
 \psline[linewidth=0.04cm,linecolor=blue,arrowsize=0.073cm 
 2.0,arrowlength=1.4,arrowinset=0.4]{->}(1.6154687,1.3632811)(1.8354688,0.6032812)
 \psline[linewidth=0.04cm,linecolor=blue,arrowsize=0.073cm 
 2.0,arrowlength=1.4,arrowinset=0.4]{<-}(1.0354688,-0.8967188)(1.2754687,-0.2967188)
 \psline[linewidth=0.04cm,linecolor=blue,arrowsize=0.073cm 
 2.0,arrowlength=1.4,arrowinset=0.4]{<-}(0.7554688,0.3632812)(0.15546876,0.82328117)
 \psline[linewidth=0.04cm,linecolor=blue,arrowsize=0.073cm 
 2.0,arrowlength=1.4,arrowinset=0.4]{->}(2.4154687,-0.5767188)(2.4554687,-1.4367187)
 \usefont{T1}{ptm}{m}{n}
 \rput(4.2,-0.20){$\Omega$}
 \rput(8.0,+1.90){$\partial \Omega \setminus \Gamma$}
 \rput(8.0,-0.45){$u\cdot n = 0$}
 \rput(8.8,-1.00){$[D(u)n+M u]_{tan} = 0$}
 \rput(0.0,-1.40){$\Gamma$}
 \end{pspicture}
 \caption{Setting of the main Navier-Stokes control problem.}
\end{figure}

We define the space $\Hspace$ as the closure in $L^2(\Omega)$ of smooth 
divergence free vector fields which are tangent to~$\BorderNav$. For
$f \in \Hspace$, we do not require that $f \cdot n = 0$ on the controlled 
boundary $\Gamma$. \fm{Of course, due to the Stokes theorem, such functions
satisfy $\int_{\Gamma} f \cdot n = 0$.}
The main result of this paper is the following \fm{small-time} global 
exact
null controllability theorem:

\begin{theorem} \label{thm.weak.null}
 Let $T > 0$ and $\uinit \in \Hspace$. There exists 
 $u \in \mathcal{C}_w^0([0,T];\Hspace) \cap L^2((0,T); H^1(\Omega))$ 
 a weak controlled trajectory (see Definition~\ref{def.weak.controlled}) 
 \fm{of}~\eqref{eq.main} satisfying  $u(0,\cdot) = \uinit$ and $u(T, \cdot) = 
 0$.
\end{theorem}

\begin{remark}
	Even though a unit dynamic viscosity is used in equation~\eqref{eq.main}, 
	Theorem~\ref{thm.weak.null} remains true for any fixed positive 
	viscosity~$\nu$ thanks to a straightforward scaling argument. Some works 
	also consider the case when the friction matrix~$M$ depends on~$\nu$
	(see \cite{MR3229096} or~\cite{MR2744916}). \fm{This} does not
	impact our proofs in the sense that we could still prove that: for any 
	$\nu > 0$, for any $T > 0$, for any smooth $M_\nu$, for any initial data 
	$\uinit$, one can find boundary controls (depending on all these 
	quantities) driving the initial data back to the null equilibrium state 
	at time $T$.
\end{remark}

\begin{remark} \label{remark.strong}
	Theorem~\ref{thm.weak.null} is stated as an existence result. The lack of 
	uniqueness both comes from the fact that multiple controls can drive the 
	initial state to zero and from the fact that it is not known whether weak 
	solutions are unique for the Navier-Stokes equation \fm{in 3D (in 2D, it is 
	known that weak solutions are unique)}. \fm{Always in the 3D case, 
    if} the initial 
	data~$\uinit$ is smooth enough, it would be interesting to know if we can 
	build a strong solution to~\eqref{eq.main} driving $\uinit$ back to zero 
    \fm{(in 2D, global existence of strong solutions is known).}
	\fm{We conjecture that building strong controlled trajectories is 
	possible}. What we do prove \fm{here} is that, if the 
	initial data~$\uinit$ is smooth enough, then our \fm{small-time} global 
	approximate null control strategy drives any weak solution starting from 
	this initial state close to zero.
\end{remark}

Although most of this paper is dedicated to the proof of 
Theorem~\ref{thm.weak.null} concerning the null controllability, we also 
explain in Section~\ref{section.trajectories} how one can adapt \fm{our} method 
to obtain \fm{small-time} global exact controllability towards any weak 
trajectory (and not only the null equilibrium state).

% ==============================================================================
\subsection{A challenging open problem as a motivation}

The \fm{small-time} global exact null controllability problem for the 
Navier-Stokes
equation was first suggested by Jacques-Louis Lions in the late 80's. It is 
mentioned in~\cite{MR1147191} in a setting where the control is a source term
supported within a small subset of the domain (this situation is similar to 
controlling only part of the boundary). In Lions' original question, the 
boundary condition on the uncontrolled part of the boundary is the Dirichlet 
boundary condition. Using our notations and our boundary control setting, the 
system considered is:
\begin{equation} \label{eq.main.dirichlet}
    \left\{
    \begin{aligned}
        \partial_t u + (u \cdot \nabla) u - \Delta u + \nabla p & = 0
        && \quad \textrm{in } \Omega, \\
        \diverg u & = 0
        && \quad \textrm{in } \Omega, \\
        u  & = 0
        && \quad \textrm{on } \BorderNav.
    \end{aligned}
    \right.
\end{equation}

\noindent 
\hypertarget{openproblem}{\textbf{Open Problem (OP)}} 
\emph{
 For any $T > 0$ and $\uinit$ in $L^2(\Omega)$ which is divergence free and 
 vanishes on $\BorderNav$, does there exist a trajectory \fm{of} 
 system~\eqref{eq.main.dirichlet} with $u(0, \cdot) = \uinit$ 
 such that $u(T,\cdot) = 0$?}
\bigskip

This is a very challenging open problem because the Dirichlet boundary 
condition gives rise to boundary layers that have a larger amplitude than 
Navier slip-with-friction boundary layers. However, we expect that the 
\fm{method} we introduce here can inspire later works on the more difficult 
case of the Dirichlet boundary condition, at least for some favorable geometric 
and functional settings.

% ==============================================================================
\subsection{Known results and related previous works} 

\subsubsection{Local results}
\label{par.local.results}

A first approach to study the controllability of Navier-Stokes systems is to 
deal with the quadratic convective term as a perturbation term and obtain 
results using the diffusive term. Of course, this kind of approach is mostly 
efficient for local results, where the quadratic term is indeed small. Most 
local proofs rely on Carleman estimates for the adjoint system. 

For the Dirichlet boundary condition, Imanuvilov proves in~\cite{MR1804497} 
\fm{small-time} local controllability to the trajectories \fm{for 2D and 3D}. 
This result has since been improved in~\cite{MR2103189} by Fern\'andez-Cara, 
Guerrero, Imanuvilov and Puel. Their proof uses Carleman estimates and weakens 
the regularity assumed on the trajectories. \fm{In particular, their proof 
implies null controllability with small initial data in~$L^{2d-2}$.}

For Navier slip-with-friction boundary conditions, two papers were published in 
2006. In~\cite{MR2268275}, the authors prove a local controllability result to 
the trajectories in 2D domains\fm{, assuming the initial data is close to the 
trajectory in $H^1$}. In~\cite{MR2224824}, Guerrero, proves \fm{small-time} 
local controllability to the trajectories for 2D and 3D domains, with 
general non-linear Navier boundary conditions. \fm{His result implies null 
controllability when the initial data is small in $H^3$. It is likely that
this result could be improved to lower this hypothesis to $L^{2d-2}$ as in the
Dirichlet case.}

\subsubsection{Global results}

The second approach goes the other way around: see the viscous term as a 
perturbation of the inviscid dynamic and try to deduce the controllability of
Navier-Stokes from the controllability of Euler. This approach is efficient to
obtain \fm{small-time} results, as inviscid effects prevail in this asymptotic.
However, if one does not control the full boundary, boundary layers appear near 
the uncontrolled boundaries~$\BorderNav$. Thus, most known results try to avoid 
this situation.

In~\cite{MR1470445}, the first author and Fursikov prove a \fm{small-time} 
global 
exact null controllability result when the domain is a manifold without border 
(in this setting, the control is a source term located in a small subset of the 
domain). Likewise, in~\cite{MR1728643}, Fursikov and Imanuvilov prove 
\fm{small-time} global exact null controllability when the control is supported 
on the 
whole boundary (i.e. $\Gamma = \partial \Omega$). In both cases, there is no 
boundary layer.

Another method to avoid the difficulties is to choose more gentle boundary
conditions. In a simple geometry (a 2D rectangular domain), Chapouly proves 
in~\cite{MR2560050} \fm{small-time} global exact null controllability for 
Navier-Stokes under the boundary condition $\nabla \times u = 0$ 
on uncontrolled boundaries. Let $[0,L]\times[0,1]$ be the considered rectangle.
Her control acts on both vertical boundaries at $x_1 = 0$ and $x_1 = L$.
Uncontrolled boundaries are the horizontal ones at $x_2 = 0$ and $x_2 = 1$.
She deduces the controllability of Navier-Stokes from \fm{the controllability} 
of Euler by linearizing around an explicit reference trajectory \fm{$u^0(t,x) :=
(h(t),0)$}, where $h$ is a smooth profile. Hence, the Euler trajectory already 
satisfies all boundary conditions and there is no boundary layer to be expected
\fm{at leading order}.

For Navier slip-with-friction boundary conditions in 2D, the first author 
proves in~\cite{MR1393067} a \fm{small-time} global approximate null 
controllability result. He proves that \fm{exact controllability can be 
achieved} in the interior of the domain. However, this is not the case near the 
boundaries. The approximate controllability is obtained in the space 
$W^{-1,\infty}$, which is not a strong enough space to be able to conclude to 
global exact null controllability using a local result. The residual boundary 
layers are too strong and have not been sufficiently handled during the control 
\fm{design} strategy.

For Dirichlet boundary conditions, Guerrero, Imanuvilov and Puel prove
in~\cite{MR2269867} (resp.~\cite{MR2994698}) for a square
(resp. a cube) where one side (resp. one face) is not controlled, 
a small time result which looks like global approximate null controllability.
Their method consists in adding a new source term (a control supported on the 
whole domain $\Omega$) to absorb the boundary layer. They prove that this 
additional control can be chosen small in $L^p((0,T);H^{-1}(\Omega))$, for $1 
< p < p_0$ (with $p_0 = 8/7$ in 2D and $4/3$ in 3D). However, this norm is too 
weak to take a limit and obtain the result stated in Open 
Problem~(\hyperlink{openproblem}{OP}) (without this fully supported 
additional control). 
Moreover, the $H^{-1}(\Omega)$ estimate seems to indicate that the role of the
inner control is to act on the boundary layer directly where it is 
located\fm{, which is somehow in contrast with the goal of achieving 
controllability with controls supported on only part of the boundary}.

All the examples detailed above tend to indicate that a new method is needed,
which fully takes into account the boundary layer in the control \fm{design} 
strategy.

\subsubsection{The "well-prepared dissipation" method}

In~\cite{MR3227326}, the second author proves \fm{small-time} global exact null 
controllability for the Burgers equation on the line segment $[0,1]$ with a
Dirichlet boundary condition at $x = 1$ (implying the presence of a boundary 
layer near the uncontrolled boundary $x = 1$). The proof relies on a method 
involving a \textit{well-prepared dissipation} of the boundary layer. The sketch
of the method is the following:
\begin{enumerate}
 \item \textbf{Scaling argument.} Let $T > 0$ be the small time given for the 
 control problem. Introduce $\varepsilon \ll 1$ a very small scale.
 Perform the usual small-time to small-viscosity fluid scaling 
 $u^\varepsilon(t,x) := \varepsilon u(\varepsilon t,x)$, yielding a new unknown 
 $u^\varepsilon$, defined on a large time scale $[0,T/\varepsilon]$,
 satisfying a vanishing viscosity equation.
 Split this large time interval in two parts: $[0,T]$ and 
 $[T,T/\varepsilon]$.
 \item \textbf{Inviscid stage.} During $[0,T]$, use (up to
 the first order) the same controls as if the system
 was inviscid. This leads to good interior controllability (far from the 
 boundaries, the system already behaves like its inviscid limit) but creates a 
 boundary layer residue near uncontrolled boundaries. 
 \item \textbf{Dissipation stage.} During the long segment $[T,T/\varepsilon]$, 
 choose null controls and let
 the system dissipate the boundary layer by itself thanks to its smoothing 
 term. 
 As $\varepsilon \rightarrow 0$, the long time scale compensates exactly for 
 the 
 small viscosity. However, as $\varepsilon\rightarrow 0$, the boundary layer 
 gets thinner and dissipates better.
\end{enumerate}
The key point in this method is to separate steps 2 and 3. Trying to control
both the inviscid dynamic and the boundary layer at the end of step 2 is too 
hard. Instead, one chooses the inviscid controls with care during step 2 in 
order to prepare the self-dissipation of the boundary layer during step 3. This
method will be used in this paper and enhanced to prove our result. In order to
apply this method, we will need a very precise description of the boundary
layers involved.

% ==============================================================================
\subsection{Boundary conditions and boundary layers for Navier-Stokes}

Physically, boundary layers are the fluid layers in the immediate vicinity of
the boundaries of a domain, where viscous effects prevail. Mathematically, they
appear when studying vanishing viscosity limits while maintaining strong 
boundary conditions. There is a huge literature about boundary conditions for 
partial differential equations and the associated boundary layers. In this 
paragraph, we give a short overview of some relevant references in our context 
for the Navier-Stokes equation.

% ==============================================================================
\subsubsection{Adherence boundary condition}
\label{par.prandtl}

The strongest and most commonly used boundary condition for Navier-Stokes is 
the full adherence (or no-slip) boundary condition $u = 0$. This condition is 
most often referred to as the Dirichlet condition although it was introduced by 
Stokes in~\cite{stokes1851effect}. Under this condition, fluid particles must 
remain at rest near the boundary. This generates large amplitude boundary 
layers.

In 1904, Prandtl proposed an equation describing the behavior of boundary layers
for this adherence condition in~\cite{prandtl1904uber}. \fm{Heuristically, 
these boundary layers are of amplitude $\mathcal{O}(1)$ and of thickness 
$\mathcal{O}(\sqrt{\nu})$ for a vanishing viscosity $\nu$.}
Although \fm{his equation has} been extensively studied, much is still to be 
learned.

Both physically and numerically, there exists situations where the boundary 
layer separates from the border: 
see~\cite{MR729416},~\cite{guyon2001hydrodynamique},~\cite{MR1051322}, 
or~\cite{MR595970}. Mathematically, it is known that solutions with 
singularities can be built~\cite{MR1476316} and that the linearized system is 
ill-posed in Sobolev spaces~\cite{MR2601044}. \fm{The equation has also been 
proved to be ill-posed in a non-linear context in~\cite{MR2849481}.
Moreover, even around explicit shear flow solutions of the Prandtl equation,
the equation for the remainder between Navier-Stokes and Euler+Prandtl is also
ill-posed (see~\cite{MR2099036} and~\cite{MR3420316}).} 

Most positive known results fall into two families. First, when the initial
data satisfies a monotonicity assumption, introduced by Oleinik 
in~\cite{MR0223134},~\cite{MR1697762}. \fm{See 
also~\cite{MR3327535},~\cite{2016arXiv160706434G},~\cite{MR3385340} 
and~\cite{MR2020656} for different proof techniques in this context.}
Second, when the initial data are analytic, it is both 
proved that the Prandtl equations are well-posed~\cite{MR1617542} and that 
Navier-Stokes converges to an Euler+Prandtl expansion~\cite{MR1617538}. For 
\fm{historical reviews} of known results, see~\cite{MR1778702} 
or~\cite{nickel1973prandtl}.
\fm{We also refer to~\cite{2016arXiv161005372M} for a comprehensive recent 
survey.}

Physically, the main difficulty is the possibility that the boundary layer
separates and penetrates into the interior of the domain (which is prevented by
the Oleinik monotonicity assumption). Mathematically, Prandtl equations lack
regularization in the tangential direction thus exhibiting a loss of derivative 
(which can be circumvented within an analytic setting). 

% ==============================================================================
\subsubsection{Friction boundary conditions}

Historically speaking, the adherence condition is posterior to another condition
stated by Navier in~\cite{navier1823memoire} which involves friction. The fluid
is allowed to slip along the boundary but undergoes friction near the 
impermeable walls. Originally, it was stated as:
\begin{equation} \label{eq.navier.alpha}
 u\cdot n = 0\quad\textrm{and}\quad \tanpart{D(u)n + \alpha u} = 0,
\end{equation}
where $\alpha$ is a scalar positive coefficient. Mathematically, $\alpha$ can
depend (smoothly) on the position and be a matrix without changing much the
nature of the estimates.

This condition has been justified from the boundary condition at the microscopic
scale in~\cite{MR988561} for the Boltzmann equation. See also~\cite{MR3027558} 
or~\cite{MR1980855} for other examples of such derivations. 

Although the adherence condition is more popular in the mathematical community,
the slip-with-friction condition is actually well suited for a large range of
applications. For instance, it is an appropriate model for turbulence near 
rough walls~\cite{launder1972lectures} or in acoustics~\cite{MR605891}. 
It is used by physicists for flat boundaries but also for curved domains 
(see~\cite{einzel1990boundary},~\cite{guo2016slip} or~\cite{panzer1992effects}).
Physically, $\alpha$ is homogeneous to $1/b$ where $b$ is a length, named 
\textit{slip length}. Computing this parameter for different situations,
both theoretically or experimentally is important for nanofluidics and polymer
flows (see~\cite{barrat1999large} or~\cite{bocquet2007flow}).

Mathematically, the convergence of the Navier-Stokes equation under the Navier
slip-with-friction condition to the Euler equation has been studied by many 
authors.
For 2D, this subject is studied in~\cite{MR1660366} and~\cite{MR2217315}.
For 3D, this subject is treated in~\cite{MR2943945} and~\cite{MR2885569}.
To obtain more precise convergence results, it is necessary to introduce 
an asymptotic expansion of the solution $u^\varepsilon$ to the vanishing
viscosity Navier-Stokes equation involving a boundary layer term.
In~\cite{MR2754340}, Iftimie and the third author prove a boundary layer 
expansion. This expansion is easier to handle than the Prandtl model because 
the main equation for the boundary layer correction is both linear and 
well-posed in Sobolev spaces. \fm{Heuristically, these boundary layers are of 
amplitude~$\mathcal{O}(\sqrt{\nu})$ and of thickness~$\mathcal{O}(\sqrt{\nu})$
for a vanishing viscosity $\nu$.}

% ==============================================================================
\subsubsection{Slip boundary conditions}

When the physical friction between the inner fluid and the solid boundary is 
very small, one may want to study an asymptotic model describing a situation
where the fluid perfectly slips along the boundary. Sadly, the perfect slip
situation is not yet fully understood in the mathematical literature.

\paragraph{2D\fm{.}} In the plane, the situation is easier. In 1969, Lions 
introduced
in~\cite{MR0259693} the free boundary condition $\omega = 0$. This condition is
actually a special case of~\eqref{eq.navier.alpha} where $\alpha$ depends on 
the position and $\alpha(x) = 2\kappa(x)$, where $\kappa(x)$ is the curvature of
the boundary at $x \in\partial\Omega$. With this condition, good convergence 
results can be obtained from Navier-Stokes to Euler for vanishing viscosities.

\paragraph{3D\fm{.}} In the space, for flat boundaries, slipping is easily 
modeled 
with the usual impermeability condition $u \cdot n = 0$ supplemented by any of 
the following equivalent conditions:
\begin{align}
  \label{eq.slip.1} \partial_n \tanpart{u} & = 0, \\
  \label{eq.slip.2} \tanpart{D(u) n} & = 0, \\
  \label{eq.slip.3} \tanpart{\rot u} & = 0.
\end{align}
For general non-flat boundaries, these conditions cease to be equivalent. 
This situation gives rise to some confusion in the literature about which 
condition correctly describes a \textit{true slip} condition. 

Formally, condition~\eqref{eq.slip.2} can be seen as the limit when 
$\alpha\rightarrow 0$ of the usual Navier slip-with-scalar-friction 
condition~\eqref{eq.navier.alpha}.
As for condition~\eqref{eq.slip.3} it can be seen as the natural extension in
3D of the 2D Lions free boundary condition. 
Let $x \in \partial\Omega$. We note $T_x$ the tangent space to 
$\partial\Omega$ at $x$. The Weingarten map (or shape operator) $\MW(x)$ at $x$ 
is defined as a linear map from $T_x$ into itself such that 
$\MW(x)\tau := \nabla_\tau n$ for any $\tau$ in $T_x$. \fm{The} image of 
$\MW(x)$ is contained in $T_x$. Indeed, since $|n|^2 = 1$ in a
neighborhood of $\partial\Omega$, $0 = \nabla_\tau (n^2) = 2 n \cdot 
\nabla_\tau n = 2 n \cdot \MW \tau$ for any $\tau$. 
\begin{lemma}[\cite{MR2674070},~\cite{MR2943945}] \label{lemma.weingarten}
If $\Omega$ is smooth, the shape operator $\MW$ is smooth. For any 
$x\in\partial\Omega$ it defines a self-adjoint operator with values in $T_x$.
Moreover, for any divergence free vector field $u$ satisfying $u\cdot n = 0$ 
on $\partial\Omega$, we have:
 \begin{equation}
   \tanpart{D(u)n + \MW u} = \frac{1}{2} (\rot u) \times n. \label{eq.m.to.rot}
 \end{equation}
\end{lemma}

Even though it is a little unusual, it seems that condition~\eqref{eq.slip.3} 
actually better describes the situation of a fluid slipping along the boundary.
The convergence of the Navier-Stokes equation to the Euler equation under this
condition has been extensively studied. In particular, let us mention the 
works by Beirao da Veiga, Crispo et al. 
(see~\cite{MR2732007},~\cite{MR2674070},~\cite{MR2784899},~\cite{MR2891190},
\cite{MR2924729},~\cite{MR2754821} and~\cite{MR2761082}),
by Berselli et al. (see~\cite{MR2761078},~\cite{MR2989457})
and by Xiao, Xin et al.
(see~\cite{MR2992041},~\cite{MR2744916},~\cite{MR2319054},~\cite{MR2805403} 
and~\cite{MR3197467}).

The difficulty comes from the fact that the Euler equation \fm{(which models 
the behavior of a perfect fluid, not subject to friction)} is only associated 
with the $u\cdot n = 0$ boundary condition \fm{for an impermeable wall}. Any 
other supplementary condition 
will be violated for some initial data. Indeed, as shown in~\cite{MR2924729}, 
even the persistence property is false for condition~\eqref{eq.slip.3} for the 
Euler equation: choosing an initial data such that~\eqref{eq.slip.3} is 
satisfied does not guarantee that it will be satisfied at time $t > 0$. 

% ==============================================================================
\subsection{Plan of the paper}

The paper is organized as follows:
\begin{itemize}
	\item
	In Section~\ref{section.slip}, we consider the special case of the slip 
	boundary condition~\eqref{eq.slip.3}. This case is easier to handle because 
	no boundary layer appears. We \fm{prove} Theorem~\ref{thm.weak.null} in this
	simpler setting in order to explain some elements of our method.
	\item
	In Section~\ref{section.friction}, we introduce the boundary layer 
	expansion that we will be using to handle the general case and we prove 
	that we can apply the well-prepared dissipation method to ensure that the 
	residual boundary layer is small at the final time.
	\item
	In Section~\ref{section.remainder}, we introduce technical terms in the 
	asymptotic expansion of the solution and we use them to carry out energy 
	estimates on the remainder. \fm{We prove
	Theorem~\ref{thm.weak.null} in the general case}.
	\item
	In Section~\ref{section.trajectories} we explain how the well-prepared 
	dissipation method detailed in the case of null controllability can be 
	adapted to prove \fm{small-time} global exact controllability to the 
	trajectories.
\end{itemize}

% ==============================================================================
\section{A special case with no boundary layer: the slip condition}
\label{section.slip}
% ==============================================================================

In this section, we consider the special case where the friction coefficient 
$M$ is the shape operator $\MW$. On the uncontrolled boundary, thanks to 
Lemma~\ref{lemma.weingarten}, the flow satisfies:
\begin{equation} \label{eq.slip}
	u \cdot n = 0 
	\quad \text{and} \quad 
	\tanpart{\rot u} = 0.
\end{equation}
In this setting, we can build an Euler trajectory satisfying this 
overdetermined boundary condition. The Euler trajectory by itself is \fm{thus} 
an excellent approximation of the Navier-Stokes trajectory, up to the boundary.
This allows us to present some elements of our method in a simple setting 
before moving on to the general case which involves boundary layers. 

As in~\cite{MR1393067}, our strategy is to deduce the controllability of the 
Navier-Stokes equation in small time from the controllability of the Euler 
equation. In order to use this strategy, we are willing to trade small time 
against small viscosity using the usual fluid dynamics scaling. \fm{Even}
in this easier context, Theorem~\ref{thm.weak.null} is new for multiply 
connected 2D domains and for all 3D domains since~\cite{MR1393067} only 
concerns simply connected 2D domains. \fm{This condition was also studied 
in~\cite{MR2560050} in the particular setting of a rectangular domain.}

% ==============================================================================
\subsection{Domain extension and weak controlled trajectories}
\label{par.extension}

We start by introducing a smooth extension $\Omext$ of our initial 
domain~$\Omega$. We choose this extended domain in such a way that $\Gamma 
\subset \Omext$ and $\BorderNav \subset \BorderExt$ (see 
Figure~\ref{fig.extension} for a simple case). This extension procedure can be
justified by standard arguments. Indeed, we already assumed that $\Omega$ is 
a smooth domain and, up to reducing the size of $\Gamma$, we can assume that its
intersection with each connected component of $\partial\Omega$ is \fm{smooth}. 
From now on, $n$ will denote the outward pointing 
normal to the extended domain $\Omext$ (which coincides with the outward 
pointing normal to $\Omega$ on the uncontrolled boundary $\BorderNav$). We will 
also need to introduce a smooth function $\varphi : \R^d \rightarrow \R$ such 
that $\varphi = 0$ on $\BorderExt$, $\varphi > 0$ \fm{in} $\Omext$ and $\varphi 
< 0$ outside of~$\bar{\Omext}$. Moreover, we assume that $\fm{|\varphi(x)|} = 
\mathrm{dist}(x, \BorderExt)$ in a small neighborhood of $\BorderExt$. Hence, 
the normal~$n$ can be computed as $- \nabla \varphi$ close to the 
boundary and extended smoothly within the full domain~$\Omext$. In the 
sequel, we will refer to~$\Omega$ as the \emph{physical domain} where we try
to build a controlled trajectory \fm{of}~\eqref{eq.main}. Things happening 
within $\Omext \setminus \Omega$ are technicalities corresponding to the choice 
of the controls and we advise the reader to focus on true physical phenomenons
happening inside $\Omega$.

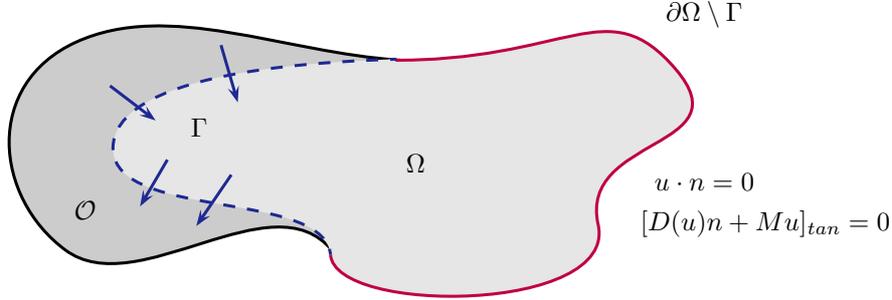
\begin{figure}[!ht] \label{fig.extension}
 \centering
 \begin{pspicture}(1.0,-2.2100782)(13.5,2.230078)
  \definecolor{gray}{rgb}{0.9,0.9,0.9}
  \definecolor{blue}{rgb}{0.11764705882352941,0.1607843137254902,0.5725490196078431}
  \definecolor{color311b}{rgb}{0.8,0.8,0.8}
  \psbezier[linewidth=0.04,fillstyle=solid,fillcolor=color311b](6.64,1.1900781)(5.1,1.290078)(3.6,1.9900781)(2.38,1.4300781)(1.16,0.8700781)(1.296471,-0.6262042)(2.24,-1.32885)(3.183529,-2.0314958)(4.96,-0.3699219)(5.74,-1.32885)
  \psbezier[linewidth=0.04,linecolor=purple,fillstyle=solid,fillcolor=gray](5.7251563,-1.3900781)(5.7851562,-2.190078)(9.5051565,-2.170078)(9.245156,-0.9900781)(8.985155,0.1899219)(11.167927,0.005063417)(10.265157,1.0699219)(9.362386,2.1347802)(8.869687,1.1633594)(6.4851565,1.1899219)
  \psbezier[linewidth=0.04,linecolor=blue,fillstyle=solid,fillcolor=gray,linestyle=dashed,dash=0.16cm 0.16cm](5.7251563,-1.3900781)(5.6451564,-0.5100781)(2.811985,-0.9486635)(2.8651562,0.0499218)(2.9183276,1.0485071)(5.6051564,1.1699218)(6.5896873,1.2033592)
  \psline[linewidth=0.04cm,linecolor=blue,arrowsize=0.073cm 2.0,arrowlength=1.4,arrowinset=0.4]{->}(4.2851562,1.3899217)(4.5051565,0.6299218)
  \psline[linewidth=0.04cm,linecolor=blue,arrowsize=0.073cm 2.0,arrowlength=1.4,arrowinset=0.4]{<-}(3.22,-0.7499219)(3.58,-0.1299219)
  \psline[linewidth=0.04cm,linecolor=blue,arrowsize=0.073cm 2.0,arrowlength=1.4,arrowinset=0.4]{<-}(3.4251564,0.3899218)(2.8251562,0.84992176)
  \psline[linewidth=0.04cm,linecolor=blue,arrowsize=0.073cm 2.0,arrowlength=1.4,arrowinset=0.4]{->}(4.42,-0.3299219)(3.96,-1.0099219)
  \usefont{T1}{ptm}{m}{n}
  \rput(6.8442187,-0.1733594){$\Omega$}
  \usefont{T1}{ptm}{m}{n}
  \rput(10.643594,1.8){$\partial \Omega \setminus \Gamma$}
  \usefont{T1}{ptm}{m}{n}
  \rput(10.644218,-0.4233594){$u\cdot n = 0$}
  \usefont{T1}{ptm}{m}{n}
  \rput(11.444219,-0.9733594){$[D(u)n+M u]_{tan} = 0$}
  \usefont{T1}{ptm}{m}{n}
  \rput(4.0,0.3){$\Gamma$}
  \usefont{T1}{ptm}{m}{n}
  \rput(2.5,-0.8){$\mathcal{O}$}
 \end{pspicture}   
 \caption{Extension of the physical domain \fm{$\Omega \subset \Omext$}.}
\end{figure}

\begin{definition} \label{def.weak.controlled}
    Let $T > 0$ and $\uinit \in \Hspace$. Let  
    $u \in \mathcal{C}_w^0([0,T];\Hspace) \cap L^2((0,T); H^1(\Omega))$.
    We will say that $u$ is a \emph{weak controlled trajectory} \fm{of} 
    system~\eqref{eq.main} with initial condition $\uinit$ when $u$ is the 
    restriction to the physical domain~$\Omega$ of a \emph{weak Leray solution} 
    in the space $\mathcal{C}_w^0([0,T];L^2(\Omext)) \cap L^2((0,T); 
    H^1(\Omext))$ on the extended domain~$\Omext$, which we still denote
    by $u$\fm{,} to:
    \begin{equation} \label{eq.main.ext}
        \left\{
        \begin{aligned}
        \partial_t u + (u \cdot \nabla) u 
        - \Delta u + \nabla p & = \force
        && \quad \textrm{in } \Omext, \\
        \diverg u & = \sigma
        && \quad \textrm{in } \Omext, \\
        u \cdot n & = 0
        && \quad \textrm{on } \BorderExt, \\
        N(u) & = 0
        && \quad \textrm{on } \BorderExt, \\
        u(0,\cdot) & = \uinit
        && \quad \textrm{\fm{in} } \Omext,
        \end{aligned}
        \right.
    \end{equation}
    where $\force \in \controlspaceforce$ is a forcing term supported 
    \fm{in $\bar{\Omext}\setminus\bar{\Omega}$}, $\sigma$ is a smooth non 
    homogeneous divergence condition also supported \fm{in 
    $\bar{\Omext}\setminus\bar{\Omega}$} and $\uinit$ has been extended 
    to~$\Omext$ such that the extension is tangent to~$\BorderExt$ and 
    satisfies the compatibility condition $\diverg \uinit = \sigma(0,\cdot)$.
\end{definition}

Allowing a non vanishing divergence outside of the physical domain is necessary
both for the control design process and because we did not restrict ourselves
to controlling initial data satisfying $\uinit \cdot n = 0$ on~$\Gamma$. 
Defining weak Leray solutions to~\eqref{eq.main.ext} is a difficult question
when one tries to obtain optimal functional spaces for the non homogeneous 
source terms. For details on this subject, we refer the reader 
to~\cite{MR2258419},~\cite{MR2866119} or~\cite{MR2679654}. In our case, since
the divergence source term is smooth, an efficient method is to start by solving 
a (stationary or evolution) Stokes problem in order to lift the non homogeneous 
divergence condition. We \fm{define $u_\sigma$ as} the solution to:
\begin{equation} \label{eq.stokes.div}
    \left\{
    \begin{aligned}
    \partial_t u_\sigma - \Delta u_\sigma + \nabla p_\sigma & = 0
    && \quad \textrm{in } \Omext, \\
    \diverg u_\sigma & = \sigma
    && \quad \textrm{in } \Omext, \\
    u_\sigma \cdot n & = 0
    && \quad \textrm{on } \BorderExt, \\
    N(u_\sigma) & = 0
    && \quad \textrm{on } \BorderExt, \\
    u_\sigma(0,\cdot) & = \uinit
    && \quad \textrm{\fm{in} } \Omext.
    \end{aligned}
    \right.
\end{equation}
Smoothness (in time and space) of $\sigma$ immediately gives smoothness on 
$u_\sigma$. These are standard maximal regularity estimates for the 
Stokes problem in the case of the Dirichlet boundary condition. For Navier
boundary conditions (sometimes referred to as Robin boundary conditions for the 
Stokes problem), we refer to~\cite{MR2325694},~\cite{MR2387243} 
or~\cite{MR2285430}. Decomposing $u = u_\sigma + u_h$,
we obtain the following system for $u_h$:
\begin{equation} \label{eq.main.ext.h}
    \left\{
    \begin{aligned}
    \partial_t u_h + (u_\sigma \cdot \nabla) u_h
    + (u_h \cdot \nabla) u_\sigma
    + (u_h \cdot \nabla) u_h
    - \Delta u_h + \nabla p_h & = \force
    - (u_\sigma \cdot \nabla) 
    u_\sigma
    && \quad \textrm{in } \Omext, \\
    \diverg u_h & = 0
    && \quad \textrm{in } \Omext, \\
    u_h \cdot n & = 0
    && \quad \textrm{on } \BorderExt, \\
    N(u_h) & = 0
    && \quad \textrm{on } \BorderExt, \\
    u_h(0,\cdot) & = 0
    && \quad \textrm{\fm{in} } \Omext.
    \end{aligned}
    \right.
\end{equation}
Defining weak Leray solutions \fm{to}~\eqref{eq.main.ext.h} is a standard 
procedure. 
They are defined as measurable functions satisfying the variational formulation 
of~\eqref{eq.main.ext.h} and some appropriate energy inequality. For in-depth
insights on this topic, we refer the reader to the classical references by 
Temam~\cite{MR1846644} or Galdi~\cite{MR1798753}. In our case, let $\ldiv$
denote the closure in $L^2\fm{(\Omext)}$ of the space of smooth divergence free 
vector
fields tangent to~$\BorderExt$. We will say that $u_h \in 
\mathcal{C}_w^0([0,T];\ldiv) \cap L^2((0,T); H^1(\Omext))$ is a weak Leray 
solution to~\eqref{eq.main.ext.h} if it satisfies the variational formulation:
\begin{equation}
    \begin{split}
        \fm{-} \iint_\Omext u_h \partial_t \phi
        & + \iint_\Omext \left( (u_\sigma \cdot \nabla) u_h
        + (u_h \cdot \nabla) u_\sigma
        + (u_h \cdot \nabla) u_h \right) \phi \\
        & + 2 \iint_\Omext D(u_h) : D(\phi)
        + 2 \iint_{\BorderExt} \tanpart{M u_h} \phi 
        = \iint_{\Omext} \left(\force
        - (u_\sigma \cdot \nabla) u_\sigma \right) \phi,
    \end{split}
\end{equation} 
for any $\phi \in \mathcal{C}^\infty_c([0,T), \bar{\Omext})$ which is 
divergence free and tangent to $\BorderExt$. We moreover require that they 
satisfy the so-called strong energy inequality for almost every $\tau < t$:
\begin{equation} \label{eq.SEI}
    \begin{split}
        \left| u_h(t) \right|_{L^2}^2
        + 4 \iint_{(\tau, t) \times \Omext} \left|D(u_h)\right|^2
        & \leq \left| u_h(\tau) \right|_{L^2}^2
        \fm{-} 4 \iint_{(\tau, t) \times \BorderExt} \tanpart{M u_h} 
        u_h  \\
        & + \iint_{(\tau, t) \times \Omext} \sigma u_h^2
        + 2(u_h \cdot \nabla) u_\sigma u_h 
        + 2\left(\force - (u_\sigma \cdot \nabla) u_\sigma 
        \right) u_h.
    \end{split}
\end{equation}
\fm{In~\eqref{eq.SEI}, the boundary term is well defined. Indeed, from the 
Galerkin method, we can obtain strong convergence of Galerkin approximations 
$u_h^n$ towards $u_h$ in $L^2((0,T);L^2(\BorderExt))$ 
(see~\cite[page~155]{MR2754340}).}

Although uniqueness of weak Leray solutions is still an open question, it is 
easy to adapt the classical Leray-Hopf theory proving global existence of weak 
solutions to the case of Navier boundary conditions (see~\cite{MR1660366} for 
2D or~\cite{MR2214949} for 3D). Once forcing terms $\force$ and $\sigma$ 
are fixed, there exists thus at least one weak Leray solution $u$ 
to~\eqref{eq.main.ext}.

In the sequel, we will mostly work within the extended domain. Our goal will 
be to explain how we choose the external forcing terms $\force$ and $\sigma$ in 
order to guarantee that the associated controlled trajectory vanishes within 
the physical domain at the final time.

% ==============================================================================
\subsection{Time scaling and small viscosity asymptotic expansion}

The global controllability time $T$ is small but fixed. Let us introduce a 
positive parameter $\varepsilon \ll 1$. We will be even more ambitious and try
to control the system during the shorter time interval $[0, \varepsilon T]$.
We perform the scaling: $\ue(t,x) := \varepsilon u(\varepsilon t, x)$ and 
$p^\varepsilon(t,x) := \varepsilon^2 p(\varepsilon t, x)$. 
\fm{Similarly, we set 
    $\force^\varepsilon(t,x):=\varepsilon^2\force(\varepsilon t,x)$
    and $\sigma^\varepsilon(t,x):=\varepsilon \sigma(\varepsilon t, x)$. }
Now, $(\ue,p^\varepsilon)$ is a solution to the following system for $t\in (0, 
T)$:
\begin{equation} \label{eq.nse.rot}
	\left\{
	\begin{aligned}
		\partial_t \ue + \left( \ue \cdot \nabla \right) \ue
		- \varepsilon \Delta \ue + \nabla p^\varepsilon 
		& = \force^\varepsilon
		&&\quad \textrm{\fm{in} } (0, T) \times \Omext, \\
		\diverg \ue & = \sigma^\varepsilon
		&&\quad \textrm{\fm{in} } (0, T) \times \Omext, \\
		\ue\cdot n & = 0 
		&&\quad \textrm{on } (0, T) \times \BorderExt, \\
		\tanpart{\rot \ue} & = 0 
		&&\quad \textrm{on } (0, T) \times \BorderExt, \\
		\ue \rvert_{t = 0} & = \varepsilon \uinit
		&&\quad \textrm{\fm{in} } \Omext.
	\end{aligned}
	\right.
\end{equation}
Due to the scaling chosen, we plan to prove that we can obtain 
$\left|\ue(T,\cdot)\right|_{L^2(\Omext)}=o(\varepsilon)$ in order to conclude
with a local result. Since $\varepsilon$ is small, we expect
$\ue$ to converge to the solution of the Euler equation. Hence, we introduce 
the following asymptotic expansion for:
\begin{align} 
 \label{eq.expansion.rot}
 \ue & = u^0 + \varepsilon u^1 + \varepsilon \Rem, \\
 \label{eq.expansion.p.rot}
 p^\varepsilon & = \fm{p^0} + \varepsilon \fm{p^1} + \varepsilon 
 \fm{\pi^\varepsilon}, \\
 \label{eq.expansion.zeta.rot}
 \force^\varepsilon & = \force^0 + \varepsilon \force^1, \\
 \label{eq.expansion.sigma.rot}
 \sigma^\varepsilon & = \sigma^0.
\end{align}
Let us provide some insight behind 
expansion~\eqref{eq.expansion.rot}-\eqref{eq.expansion.sigma.rot}. The 
first term~$(u^0,p^0\fm{,\force^0,\sigma^0})$ is \fm{the solution to a 
controlled} 
Euler 
equation. It models a smooth reference trajectory around which we are 
linearizing the Navier-Stokes 
equation. This trajectory will be chosen in such a way that it flushes the 
initial data out of the domain in time $T$. The second 
term~$(u^1,p^1,\force^1)$ takes 
into account the initial data~$\uinit$, which will be flushed out of the physical 
domain by the flow~$u^0$. Eventually,~$\fm{(\Rem,\pi^\varepsilon)}$ contains 
higher order residues.
We need to prove $\left|\Rem(T,\cdot)\right|_{L^2(\Omext)}=o(1)$ in order to 
be able to conclude the proof of Theorem~\ref{thm.weak.null}.

% ==============================================================================
\subsection{A return method trajectory for the Euler equation}
\label{par.euler}

At order~$\mathcal{O}(1)$, the first part~$(u^0, p^0)$ of our expansion is a 
solution to the Euler equation. Hence, the pair~$(u^0, p^0)$ is a 
\textit{return-method-like} trajectory of the Euler equation on~$(0,T)$: 
\begin{equation} \label{eq.euler.ext}
	\left\{
		\begin{aligned}
			\partial_t u^0 + \left( u^0 \cdot \nabla \right) u^0 + \nabla p^0
			& = \force^0
			&&\quad \textrm{in } (0, T) \times \Omext, \\
			\diverg u^0 & = \sigma^0
			&&\quad \textrm{in } (0, T) \times \Omext, \\
			u^0 \cdot n & = 0 
			&&\quad \textrm{on } (0, T) \times \partial\Omext, \\
			u^0(0, \cdot) & = 0 
			&&\quad \textrm{in } \Omext, \\
			u^0(T, \cdot) & = 0 
			&&\quad \textrm{in } \Omext,
		\end{aligned}
	\right.
\end{equation}
where $\force^0$ and $\sigma^0$ are smooth forcing terms supported \fm{in 
$\bar{\Omext}\setminus\bar{\Omega}$}. We want to use this reference trajectory 
to flush any particle outside of the physical domain within the fixed time 
interval $[0,T]$. Let us introduce the flow~$\Phi^0$ associated with~$u^0$:
\begin{equation} \label{eq.flow}
 \left\{
  \begin{aligned}
   \Phi^0(t,t,x) & = x, \\
   \partial_s \Phi^0(t,s,x) & = u^0(s, \Phi^0(t, s, x)).
  \end{aligned}
 \right.
\end{equation}
Hence, we look for trajectories satisfying:
\begin{equation} \label{eq.flow.out}
 \fm{\forall x \in \bar{\Omext}, 
 \exists t_x \in (0,T),
 \quad \Phi^0\left(0,t_x,x\right) \notin \bar{\Omega}}.
\end{equation}
We do not require that the time $t_x$ be the same for all $x \in \Omext$. 
Indeed, it might not be possible to flush all of the points outside of the 
physical domain at the same time. \fm{Property~\eqref{eq.flow.out} is obvious
for points $x$ already located in $\bar{\Omext}\setminus\bar{\Omega}$. For 
points lying within the physical domain, we use:}

\begin{lemma} \label{lemma.euler}
	There exists a solution~$(u^0, p^0, \force^0, \sigma^0) \in 
	\mathcal{C}^\infty([0,T]\times\bar{\Omext},\R^d \times \R \times \R^d 
	\times \R)$ 
	to system~\eqref{eq.euler.ext} such that the flow~$\Phi^0$ defined 
	in~\eqref{eq.flow} satisfies~\eqref{eq.flow.out}. Moreover, $u^0$ can be
	chosen such that:
	\begin{equation} \label{eq.euler.irrot}
	  \nabla \times u^0 = 0 \quad \textrm{in} \quad [0,T] \times 
      \fm{\bar{\Omext}}.
	\end{equation}
    \fm{Moreover, $(u^0, p^0, \force^0, \sigma^0)$ are compactly supported in 
    $(0,T)$. In the sequel, when we need it, we will implicitly extend them by 
    zero after $T$.}
\end{lemma}

This lemma is the key argument of multiple papers concerning the 
\fm{small-time} 
global exact controllability of Euler equations. We refer to the following
references for detailed statements and construction of these reference 
trajectories. First, the first author used it in~\cite{MR1233425} for 2D
simply connected domains, then in~\cite{MR1380673} for general 2D domains 
\fm{when}
$\Gamma$ intersects all connected components of $\partial\Omega$.
Glass adapted the argument for 3D domains (when $\Gamma$ intersects all
connected components of the boundary), for simply connected domains 
in~\cite{MR1485616} then for general domains in~\cite{MR1745685}.
He also used similar arguments to study the obstructions to approximate 
controllability in 2D when $\Gamma$ does not intersect all connected 
components of the boundary for general 2D domains in~\cite{MR1860818}.
Here, we use the assumption that our control domain~$\Gamma$ intersects all
connected parts of the boundary $\partial\Omega$. The fact that 
condition~\eqref{eq.euler.irrot} can be achieved is a direct consequence of
the construction of the reference profile $u^0$ as a potential flow:
$u^0(t,x) = \nabla \theta^0(t,x)$, where $\theta^0$ is smooth.

% ==============================================================================
\subsection{Convective term and flushing of the initial data}

We move on to order $\mathcal{O}(\varepsilon)$. Here, the initial data $\uinit$
comes into play. We build $u^1$ as the solution to:
\begin{equation} \label{eq.u1.ext}
	\left\{
		\begin{aligned}
			\partial_t u^1 + \left( u^0 \cdot \nabla \right) u^1 
			+ \left( u^1 \cdot \nabla \right) u^0 + \nabla p^1
			& = \Delta u^0 + \force^1
			&&\quad \textrm{in } (0, T) \times \Omext, \\
			\diverg u^1 & = 0
			&&\quad \textrm{in } (0, T) \times \Omext, \\
			u^1 \cdot n & = 0 
			&&\quad \textrm{on } (0, T) \times \partial\Omext, \\
			u^1(0, \cdot) & = \uinit
			&&\quad \textrm{in } \Omext,
		\end{aligned}
	\right.
\end{equation}
where~$\force^1$ is a forcing term supported 
in~$\bar{\Omext}\setminus\bar{\Omega}$. Formally, 
equation~\eqref{eq.u1.ext} also takes into account a residual 
term~$\Delta u^0$. Thanks to~\eqref{eq.euler.irrot}, we have 
$\Delta u^0 = \nabla (\diverg u^0) = \nabla \sigma^0$. It is thus 
smooth, supported in~$\bar{\Omext}\setminus\bar{\Omega}$ and can 
be canceled by incorporating it into $\force^1$.
The following lemma is natural thanks to the 
choice of a good flushing trajectory $u^0$:

\begin{lemma} \label{lemma.u1}
 Let $\uinit \in H^3(\Omext) \cap \ldiv$. There exists a force
 $\force^1 \in \mathcal{C}^1([0,T],H^1(\Omext)) \cap 
 \mathcal{C}^0([0,T],H^2(\Omext))$ such that the associated solution~$u^1$ to
 system~\eqref{eq.u1.ext} satisfies $u^1(T,\cdot) = 0$. Moreover, $u^1$ is 
 bounded in $L^\infty((0,T),H^3(\Omext))$.
 In the sequel, it is implicit that we extend
 $(u^1, p^1, \force^1)$ by zero after $T$.
\end{lemma}

This lemma is mostly a consequence of the works on the Euler equation,
already mentioned in the previous paragraph, due to the first author in 2D, then 
to Glass in 3D. However, in these original works, the regularity obtained for the 
constructed trajectory would not be sufficient in our context. Thus, we provide
in Appendix~\ref{annex.euler} a constructive proof which enables us
to obtain the regularity for~$\force^1$ and~$u^1$ stated in
Lemma~\ref{lemma.u1}. We only give here a short overview of the main
idea of the proof. The interested reader can also start with the nice 
introduction given by Glass in~\cite{MR1660528}.

The intuition behind the possibility to control $u^1$ is to 
introduce $\omega^1 := \rot u^1$ and to write~\eqref{eq.u1.ext} in 
vorticity form, within the physical domain~$\Omega$:
 \begin{equation} \label{eq.w1}
  \left\{
   \begin{aligned}
    \partial_t \omega^1 + \left( u^0 \cdot \nabla \right) \omega^1 
    - \left( \omega^1 \cdot \nabla \right) u^0 & = 0
    &&\quad \textrm{in } (0,T) \times \Omega, \\
    \omega^1(0, \cdot) & = \rot \uinit
    &&\quad \textrm{in } \Omega.
   \end{aligned}
  \right.
 \end{equation}
The term $\left( \omega^1 \cdot \nabla \right) u^0$ is specific to 
the 3D setting and does not appear in 2D (where the vorticity is merely 
transported). Nevertheless, even in 3D, the support of the vorticity is 
transported \fm{by $u^0$}. Thus, thanks to hypothesis~\eqref{eq.flow.out}, 
$\omega^1$ will vanish inside $\Omega$ at time $T$ provided that we choose 
null boundary conditions for $\omega^1$ on the controlled boundary 
$\Gamma$ when the characteristics enter in the physical domain. 
Hence, we can build a trajectory such that 
$\fm{\omega^1}(T,\cdot) = 0$ inside $\Omega$. 
Combined with the divergence free condition and null boundary data,
this yields that $u^1(T,\cdot) = 0$ inside $\Omega$, at least for 
simple geometries.

% ==============================================================================
\subsection{Energy estimates for the remainder}
\label{par.r.estimate}

In this paragraph, we study the remainder defined in 
expansion~\eqref{eq.expansion.rot}. We write the equation for the remainder in 
the extended domain $\Omext$:
\begin{equation} \label{eq.R.rot}
	\left\{
	\begin{aligned}
		\partial_t \Rem 
		+ \left( \ue \cdot \nabla \right) \Rem
		- \varepsilon \Delta \Rem
		+ \nabla \pi^\varepsilon 
		& = f^\varepsilon - A^\varepsilon \Rem,
		&&\quad \textrm{in } (0,T)\times\Omext, \\
		\diverg \Rem & = 0
		&&\quad \textrm{in } (0,T)\times\Omext, \\
		\tanpart{\rot \Rem} 
		& = - \tanpart{\rot u^1}
		&&\quad \textrm{\fm{on} } (0,T)\times\partial\Omext, \\
		\Rem \cdot n & = 0
		&&\quad \textrm{\fm{on} } (0,T)\times\partial\Omext, \\
		\Rem(0, \cdot) & = 0
		&&\quad \textrm{in } \Omext,
	\end{aligned}
	\right.
\end{equation}
where we used the notations:
\begin{align}
 A^\varepsilon \Rem & := (\Rem\cdot\nabla)\left(u^0 + \varepsilon u^1\right), \\
 f^\varepsilon & := \varepsilon \Delta u^1 \fm{-} \varepsilon 
 (u^1\cdot\nabla)u^1.
\end{align}
We want to establish a standard $L^\infty(L^2)\cap L^2(H^1)$ energy estimate for
the remainder. As usual, formally, we multiply equation~\eqref{eq.R.rot} 
by $\Rem$ and integrate by parts. Since we are considering weak solutions, 
some integration by parts may not be justified because we do not have enough
regularity to give them a meaning. However, the usual technique applies: one
can recover the estimates obtained formally from the variational formulation
of the problem, the energy equality for the first terms of the expansion and 
the energy inequality of the definition of weak solutions 
(see~\cite[page 168]{MR2754340} for an example of such an argument).
We proceed term by term:
\begin{align}
 \label{eq.ibp.1}
 \int_{\Omext} \partial_t \Rem \cdot \Rem 
 & = \frac{1}{2}\ddt\int_{\Omext}\left|\Rem\right|^2, \\
 \label{eq.ibp.2}
 \int_{\Omext} \left( u^\varepsilon \cdot \nabla \right) \Rem \cdot \Rem 
 & = - \frac{1}{2} \int_{\Omext} \left(\diverg u^\varepsilon\right) 
 \left|\Rem\right|^2, \\
 \label{eq.ibp.3}
 - \varepsilon \int_{\Omext} \Delta \Rem \cdot \Rem 
 & = \varepsilon \int_{\Omext} \left| \nabla \times \Rem \right|^2
 - \varepsilon \int_{\partial\Omext} \fm{\left(\Rem \times (\nabla \times \Rem) 
 \right)} \cdot n, \\
 \label{eq.ibp.4}
 \int_{\Omext} \nabla \pi^\varepsilon \cdot \Rem 
 & = 0.
\end{align}
In~\eqref{eq.ibp.2}, we will use the fact that $\diverg u^\varepsilon = 
\diverg u^0 = \sigma^0$ is
\fm{known and bounded independently of $\Rem$}. In~\eqref{eq.ibp.3}, we use 
the boundary condition on $\Rem$ to estimate the boundary term:
\begin{equation} \label{eq.r.estimate.bt.shape}
 \begin{split}
  \left| \int_{\partial\Omext} \Rem \times (\nabla \times \Rem) \cdot n 
  \right|
  & = \left| \int_{\partial\Omext} \Rem \times (\nabla \times u^1) \cdot n 
  \right| \\
  & = \left| \int_{\Omext} \diverg \left( \Rem \times \omega^1 \right) 
  \right| \\
  & = \left| \int_{\Omext} (\nabla \times \Rem) \cdot \omega^1
  - \Rem \cdot (\nabla \times \omega^1) \right| \\
  & \leq \frac{1}{2} \int_{\Omext} \left|\nabla\times\Rem\right|^2
  +\frac{1}{2} \int_{\Omext} \left|\omega^1\right|^2
  +\frac{1}{2} \int_{\Omext} \left|\Rem\right|^2
  +\frac{1}{2} \int_{\Omext} \left|\nabla\times\omega^1\right|^2.
 \end{split}
\end{equation}
We split the forcing term estimate as:
\begin{equation} \label{eq.split.f}
  \left| \int_{\Omext} f^\varepsilon \cdot \Rem \right|
  \leq \fm{\frac{1}{2}} \left|f^\varepsilon\right|_{2} \left( 1 
  + \left|\Rem\right|_{2}^2\right).
\end{equation}
Combining 
estimates~\eqref{eq.ibp.1}-\eqref{eq.ibp.4},~\eqref{eq.r.estimate.bt.shape} 
and~\eqref{eq.split.f} yields:
\begin{equation} \label{eq.R.rot.gron1}
  \ddt |\Rem|^2_2 
  + \varepsilon |\nabla \times \Rem|^2_2 \leq 
  \left(\fm{2} \varepsilon \left|u^1\right|_{H^2}^2
  + \left|f^\varepsilon\right|_2 \right)
  + \left(\varepsilon + \left|\sigma^0\right|_\infty 
  + \fm{2} \left|A^\varepsilon\right|_\infty + 
  \left|f^\varepsilon\right|_{2}\right) 
  |\Rem|^2_2.
\end{equation}
Applying Grönwall's inequality 
by integrating over $(0,T)$ and using the null initial condition gives:
\begin{equation} \label{eq.R.rot.gron2}
 \left|\Rem\right|^2_{L^\infty(L^2)} + 
 \varepsilon \left|\nabla \times \Rem\right|^2_{L^2(L^2)}
 = \mathcal{O}(\varepsilon).
\end{equation}
This paragraphs proves that, once the source terms $\force^\varepsilon$ and 
$\sigma^\varepsilon$ are fixed \fm{as above}, any weak Leray solution 
to~\eqref{eq.nse.rot} \fm{is small at the final time. Indeed, thanks} to 
Lemma~\ref{lemma.euler} and Lemma~\ref{lemma.u1}, $u^0(T) = u^1(T) = 0$. 
At the final time,~\eqref{eq.R.rot.gron2} gives:
\begin{equation}
 \left| \ue(T,\cdot) \right|_{L^2(\Omext)}
 \leq \varepsilon \left| \Rem(T,\cdot) \right|_{L^2(\Omext)}
 = \mathcal{O}(\varepsilon^{3/2}).
\end{equation}

% ==============================================================================
\subsection{Regularization and local arguments}
\label{par.proof}

In this paragraph, we explain how to chain our arguments in order to prove
Theorem~\ref{thm.weak.null}. We will need to use a local argument to finish
bringing the velocity field exactly to the null equilibrium state \fm{(see 
paragraph~\ref{par.local.results} for references on null controllability
of Navier-Stokes)}:

\begin{lemma}[\cite{MR2224824}] \label{lemma.local}
 Let $T > 0$. There exists $\delta_T > 0$ such that, for any $\uinit \in 
 H^3(\fm{\Omext})$ which is divergence free, tangent to $\BorderExt$, satisfies
 the compatibility assumption $N(\uinit) = 0$ on $\BorderExt$ and of size
 $\left|\uinit\right|_{H^3(\fm{\Omext})} \leq \delta_T$, there exists a control 
 $\force \in \controlspaceforce$ supported outside of $\bar{\Omega}$ such that 
 the strong solution to~\eqref{eq.main.ext} with $\sigma = 0$ satisfies 
 $u(T,\cdot) = 0$.
\end{lemma}

\fm{In} this context of small initial data, the existence and uniqueness 
of a strong solution is proved in~\cite{MR2224824}. We also use the following 
smoothing lemma for our Navier-Stokes system:

\begin{lemma} \label{lemma.regularization}
    Let $T > 0$. There exists a continuous function $C_T$ with $C_T(0) = 0$, 
    such that, if $\uinit \in \ldiv$ and $u \in \mathcal{C}_w^0([0,T];\ldiv) 
    \cap L^2((0,T); H^1(\Omext))$ is a weak Leray solution 
    to~\eqref{eq.main.ext}, with $\force = 0$ and $\sigma = 0$:
    \begin{equation}
        \exists t_u \in [0,T], \quad \left|u(t_u,\cdot)\right|_{H^3(\Omext)}
        \leq C_T \left(\left|\uinit\right|_{L^2(\Omext)}\right).
    \end{equation}
\end{lemma}
\begin{proof}
 This result is proved by Temam in~\cite[Remark 3.2]{MR645638} in the harder 
 case of Dirichlet boundary condition. His method can be adapted to the Navier 
 boundary condition and one could track down the constants to explicit the 
 shape of the function $C_T$. For the sake of completeness, we provide a 
 standalone proof in a slightly more general context (see 
 Lemma~\ref{lemma.regularization.traj}, Section~\ref{section.trajectories}).
\end{proof}

We can now explain how we combine these arguments to prove 
Theorem~\ref{thm.weak.null}. Let $T > 0$ be the allowed control time and 
$\uinit \in \fm{\Hspace}$ the (potentially large) initial data to be 
controlled. The proof of Theorem~\ref{thm.weak.null} follows the following 
steps:
\begin{itemize}
	\item We start by extending $\Omega$ into $\Omext$ as explained in 
	paragraph~\ref{par.extension}. We also extend the initial data~$\uinit$ to 
	all of $\Omext$, still denoting it by~$\uinit$. We choose an extension
	such that~$\uinit\cdot n = 0$ on~$\BorderExt$ and 
	$\sigma_* := \diverg \uinit$ is smooth (and supported in 
	$\Omext \setminus \Omega$). We start with a short preparation phase where we
	let $\sigma$ decrease from its initial value to zero, relying on the existence 
	of a weak solution once a smooth $\sigma$ profile is fixed, say 
	$\sigma(t,x) := \beta(t) \sigma_*$, where $\beta$ smoothly decreases
	from $1$ to $0$. Then, once the data is divergence free, we use 
	Lemma~\ref{lemma.regularization} to deduce the existence of a time 
	$T_1 \in (0,T/4)$ such that $u(T_1,\cdot) \in H^3(\Omext)$. This is why
	we can assume that the new "initial" data has $H^3$ regularity and is
	divergence free. We can thus apply Lemma~\ref{lemma.u1}.
	\item Let $T_2 := T/2$. Starting from this new smoother initial data 
	$u(T_1,\cdot)$, we proceed with the	small-time global approximate 
	controllability method explained above on a time interval of size 
	$T_2-T_1\geq T/4$. For any $\delta > 0$, we know that we can build \fm{a}
	trajectory starting from $u(T_1,\cdot)$ and such that $u(T_2,\cdot)$ is 
	smaller than $\delta$ in $L^2(\Omext)$. It particular, it can be made
	small enough such that $C_{\frac{T}{4}}(\delta) \leq \delta_{\frac{T}{4}}$, 
	where $\delta_{\frac{T}{4}}$ comes from Lemma~\ref{lemma.local} and
	the function $C_{\frac{T}{4}}$ comes from Lemma~\ref{lemma.regularization}.
	\item Repeating the regularization argument of 
	Lemma~\ref{lemma.regularization}, we deduce the existence of
	a time $T_3\in\left(\frac{T}{2}, \frac{3T}{4}\right)$ such that 
	$u(T_3,\cdot)$ is smaller than $\delta_{\frac{T}{4}}$ in $H^3(\Omext)$.
	\item We use Lemma~\ref{lemma.local} on the time interval $\left[T_3, T_3 + 
	\frac{T}{4}\right]$ to reach exactly zero. Once the system is at rest, it 
	stays there until the final time $T$.
\end{itemize}

This concludes the proof of Theorem~\ref{thm.weak.null} in the case of the slip
condition. For the general case, we will use the same proof skeleton, but we 
will need to control the boundary layers. In the following sections, we explain 
how we can obtain \fm{small-time} global approximate null controllability in the
general case.

% ==============================================================================
\section{Boundary layer expansion and dissipation} 
\label{section.friction}
% ==============================================================================

As in the previous section, the \fm{allotted} physical control time $T$ is 
fixed (and
potentially small). We introduce an arbitrary mathematical time scale 
$\varepsilon \ll 1$ and we perform the usual scaling 
$\ue(t,x) := \varepsilon u(\varepsilon t, x)$ and 
$p^\varepsilon(t,x) := \varepsilon^2 p(\varepsilon t, x)$. 
In this harder setting involving a boundary layer expansion, we do not try to
achieve approximate controllability towards zero in the smaller physical time 
\fm{interval} $[0,\varepsilon T]$ like it was possible to do in the previous 
section. 
Instead, we will use the virtually long mathematical time interval to dissipate 
the boundary layer. Thus, we consider $(\ue,p^\varepsilon)$ the solution to:
\begin{equation} \label{eq.nse}
 \left\{
 \begin{aligned}
   \partial_t \ue + \left( \ue \cdot \nabla \right) \ue
   - \varepsilon \Delta \ue + \nabla p^\varepsilon & = \force^\varepsilon
   &&\quad \textrm{in } (0, T/\varepsilon) \times \Omext, \\
   \diverg \ue & = \sigma^\varepsilon
   &&\quad \textrm{in } (0, T/\varepsilon) \times \Omext, \\
   \ue \cdot n & = 0 
   &&\quad \textrm{on } (0, T/\varepsilon) \times \BorderExt, \\
   N(\ue) & = 0 
   &&\quad \textrm{on } (0, T/\varepsilon) \times \BorderExt, \\
   \ue \rvert_{t = 0} & = \varepsilon \uinit
   &&\quad \textrm{in } \Omext.
 \end{aligned}
 \right.
\end{equation}
Here again, we do not expect to reach exactly zero \fm{with this part of the 
strategy}. However, 
we would like to build a sequence of solutions such that 
$\left|u\left(T,\cdot\right)\right|_{L^2(\Omext)}=o(1)$.
As in Section~\ref{section.slip}, this will allow us to apply a local result 
with a small initial data, a fixed time and a fixed viscosity. 
Due to the scaling chosen, this conditions translates into proving that  
$\left|\ue\left(\frac{T}{\varepsilon},\cdot\right)\right|_{L^2(\Omext)}=o(\varepsilon)$.
Following \fm{and enhancing} the original boundary layer expansion for Navier 
slip-with-friction boundary conditions proved by Iftimie and the third author 
in~\cite{MR2754340}, we introduce the following expansion:
\begin{align} 
 \label{eq.expansion}
 \ue(t,x) 
 & = u^0(t,x) 
 + \sqrt{\varepsilon} v\left(t,x, \frac{\varphi(x)}{\sqrt{\varepsilon}}\right)
 + \varepsilon u^1(t,x) 
 + \ldots
 + \varepsilon \Rem(t,x), \\
 \label{eq.expansion.p}
 p^\varepsilon(t,x) 
 & = p^0(t,x) 
 + \varepsilon p^1(t,x)
 + \ldots 
 + \varepsilon \pi^\varepsilon(t,x).
\end{align}
\fm{The forcing terms are expanded as:}
\begin{align}
 \force^\varepsilon(t,x) 
 & = \force^0(t,x) 
 + \sqrt{\varepsilon} \force^v\left(t,x, \frac{\varphi(x)}{\sqrt{\varepsilon}}\right)
 + \varepsilon \force^1(t,x), \\
 \sigma^\varepsilon(t,x)
 & = \sigma^0(t,x).
\end{align}
Compared with expansion~\eqref{eq.expansion.rot}, 
expansion~\eqref{eq.expansion} introduces a boundary correction $v$.
Indeed, $u^0$ does not satisfy the Navier slip-with-friction 
boundary condition on $\BorderExt$. The purpose of the second term~$v$ is to 
recover this boundary condition by introducing the tangential boundary layer 
generated by~$u^0$. In equations~\eqref{eq.expansion} 
and~\eqref{eq.expansion.p},
the missing terms are technical terms which will help us prove that the
remainder is small. We give the details of this technical part in
Section~\ref{section.remainder}. We use the same profiles $u^0$ and $u^1$ 
as in the previous section (extended by zero after $T$). \fm{Hence, 
$u^\varepsilon \approx \sqrt{\varepsilon} v$ after $T$ and we must understand
the behavior of this boundary layer residue that remains after the short
inviscid control strategy.}

% ==============================================================================
\subsection{Boundary layer profile equations}
\label{par.def.v}

Since the Euler system is a first-order system, we have only been able to 
impose a single scalar boundary condition in~\eqref{eq.euler.ext} (namely, 
$u^0\cdot n=0$ on $\BorderExt$). Hence, the full Navier slip-with-friction
boundary condition is not satisfied by $u^0$. Therefore, at order
$\mathcal{O}(\sqrt{\varepsilon})$, we introduce a tangential boundary layer 
correction~$v$. This profile is expressed in terms both 
of the slow space variable $x \in \Omext$ and a fast scalar variable 
$z = \varphi(x)/\sqrt{\varepsilon}$. As in~\cite{MR2754340}, $v$ is the 
solution to:
\begin{equation} \label{eq.v}
 \left\{
 \begin{aligned}
  \partial_t v + \tanpart{(u^0 \cdot \nabla) v + (v \cdot \nabla) u^0}
   + u^0_\flat z \partial_z v - \partial_{zz} v & = \force^v
   &&\quad \textrm{in } \R_+ \times \fm{\bar{\Omext} \times \R_+}, \\
   \partial_z v(t, x, 0) & = g^0(t,x) 
   &&\quad \textrm{in } \R_+ \times \fm{\bar{\Omext}}, \\
   v(0, \fm{x, z}) & = 0 
   &&\quad \textrm{in } \fm{\bar{\Omext}\times \R_+}, \\
 \end{aligned}
 \right.
\end{equation}
where we introduce the following definitions:
\begin{align} 
 \label{eq.u0flat}
 u^0_\flat(t,x) & \fm{:=} - \frac{u^0(t,x) \cdot n(x)}{\varphi(x)}
 & \textrm{in } \R_+ \times \Omext, \\
 \label{eq.def.g0}
 g^0(t,x) & \fm{:=} 2 \chi(x) N(u^0)(t,x)
 & \textrm{in } \R_+ \times \Omext.
\end{align}
Unlike in~\cite{MR2754340}, we introduced an inhomogeneous source 
term~$\force^v$ in~\eqref{eq.v}. This corresponds to a smooth control term 
\fm{whose} $x$-support is located within $\bar{\Omext} \setminus \bar{\Omega}$. 
Using the transport term, this outside control will enable us to modify the 
behavior of $v$ inside the physical domain $\Omega$. \fm{Let us state the 
following points about equations~\eqref{eq.v},~\eqref{eq.u0flat} 
and~\eqref{eq.def.g0}:}
\begin{itemize}
    \item \fm{The boundary layer profile depends on $d+1$ spatial variables 
    ($d$ 
    slow variables $x$ and one fast variable $z$) and is thus not set in 
    curvilinear coordinates. This approach used in~\cite{MR2754340} lightens 
    the computations. It is implicit that $n$ actually refers to the extension 
    $-\nabla\varphi$ of the normal (as explained in 
    paragraph~\ref{par.extension}) and that this extends 
    formulas~\eqref{eq.def.tanpart} defining the tangential part of a vector
    field and~\eqref{eq.def.navier} defining the Navier operator 
    inside~$\Omext$.}
    \item \fm{The boundary profile is tangential, even inside the domain.
    For any $x\in \bar{\Omext}$, $z \geq 0$ and $t\geq 0$,
    we have $v(t,x,z) \cdot n(x) = 0$. It is easy to check that, as soon as
    the source term $\force^v \cdot n = 0$, the evolution equation~\eqref{eq.v} 
    preserves the relation $v(0,x,z)\cdot n(x) = 0$ of the initial time. This
    orthogonality property is the reason why equation~\eqref{eq.v} is linear.
    Indeed, the quadratic term $(v\cdot n) \partial_z v$ should have been taken
    into account if it did not vanish. In the sequel, we will check that our
    construction satisfies the property $\force^v \cdot n = 0$.}
    \item In~\eqref{eq.def.g0}, we introduce a smooth \fm{cut-off} function 
    $\chi$, 
    satisfying $\chi = 1$ on $\BorderExt$. This is intended to help us 
    guarantee that $v$ is compactly supported near $\BorderExt$, while ensuring
    that $v$ compensates the Navier slip-with-friction boundary trace of~$u^0$.
    \fm{See paragraph~\ref{par.localize} for the choice of~$\chi$.}
    \item Even though $\varphi$ vanishes on~$\fm{\BorderExt}$, $u^0_\flat$ is 
    not 
    singular near the boundary because of the impermeability condition 
    $u^0\cdot n= 0$. Since $u^0$ is smooth, a Taylor expansion proves
    that $u^0_\flat$ is smooth in~$\bar{\Omext}$.
\end{itemize}

% ==============================================================================
\subsection{Large time asymptotic decay of the boundary layer profile}

In the previous paragraph, we defined the boundary layer profile through 
equation~\eqref{eq.v} for any $t \geq 0$. Indeed, we will need this expansion 
to hold on the large time interval $[0,T/\varepsilon]$. Thus, we prefer to 
define it directly for any $t \geq 0$ in order to stress out that this boundary
layer profile does not depend in any way on~$\varepsilon$. Is it implicit that, 
for $t \geq T$, the Euler reference flow~$u^0$ is extended by~$0$. Hence, for 
$t \geq T$, system~\eqref{eq.v} reduces to a parametrized heat equation on the 
half line $z \geq 0$ (where the slow variables $x \in \Omext$ play the role 
of parameters):
\begin{equation} \label{eq.v.after}
  \left\{
  \begin{aligned}
    \partial_t v - \partial_{zz} v & = 0,
    &&\quad \textrm{in } \R_+ \times \Omext, \quad \textrm{for } t \geq T, \\
    \partial_z v(t, x, 0) & = 0 
    &&\quad \textrm{in } \{0\} \times \Omext, \quad \textrm{for } t \geq T.
  \end{aligned}
  \right.
\end{equation}
The behavior of the solution to~\eqref{eq.v.after} depends on its ``initial''
data $\bar{v}(x,z) := v(T,x,z)$ at time $T$. Even without any assumption 
on~$\bar{v}$, this heat system exhibits smoothing properties and dissipates 
towards the null equilibrium state. It can for example be proved that:
\begin{equation} \label{eq.relax}
 \left| v(t, x, \cdot) \right|_{L^2(\R_+)}
 \lesssim
 t^{-\frac{1}{4}} \left| \bar{v}(x, \cdot) \right|_{L^2(\R_+)}.
\end{equation}
However, as the equation is set on the half-line $z \geq 0$, the energy decay 
obtained in~\eqref{eq.relax} is rather slow. Moreover, without any additional
assumption, this estimate cannot be improved. It is indeed standard to prove
asymptotic estimates for the solution $v(t,x,\cdot)$ involving the 
corresponding 
Green function (see~\cite{MR2727993},~\cite{MR1183805}, or~\cite{MR571048}). 
Physically, this is due to the fact that the average of $v$ is preserved under 
its evolution by equation~\eqref{eq.v.after}. The energy contained by low 
frequency modes decays slowly. Applied at the final time $t = T/\varepsilon$,
estimate~\eqref{eq.relax} yields:
\begin{equation} \label{eq.relax.v.naive}
 \left| \sqrt{\varepsilon} 
 v\left(\frac{T}{\varepsilon},\cdot,\frac{\varphi(\cdot)}
 {\sqrt{\varepsilon}}\right) \right|_{L^2(\Omext)}
 = \mathcal{O}\left(\varepsilon^{\frac{1}{2}+\frac{1}{4}+\frac{1}{4}}\right),
\end{equation}
where the last $\varepsilon^{\frac{1}{4}}$ factor comes from the Jacobian of
the fast variable scaling (see~\cite[Lemma~3, page~150]{MR2754340}). Hence, the 
natural decay $\mathcal{O}(\varepsilon)$ obtained in~\eqref{eq.relax.v.naive}
is not sufficient to provide an asymptotically small boundary layer residue in 
the physical scaling. After division by $\varepsilon$, we only obtain a
$\mathcal{O}(1)$ estimate. This motivates the fact that we need to design a 
control strategy to enhance the natural dissipation of the boundary layer 
residue after the main inviscid control step is finished.

Our strategy will be to guarantee that $\bar{v}$ satisfies a finite number of 
vanishing moment conditions for $k \in \N$ of the form:
\begin{equation} \label{eq.moments.k}
 \forall x \in \Omext, \quad \int_{\R_+} z^k \bar{v}(x,z) \dz = 0.
\end{equation}
\fm{These} conditions also correspond to vanishing derivatives at zero 
for the Fourier transform in~$z$ of~$\bar{v}$ (or its even extension to~$\R$). 
If we succeed to kill enough moments in the boundary layer at the end of the 
inviscid phase, we can obtain arbitrarily good polynomial decay properties. 
For~$s, n \in \N$, let us introduce the following weighted Sobolev spaces:
\begin{equation} \label{eq.def.h2n}
 H^{s,n}(\R) :=
 \left\{ f \in H^s(\R), \quad 
 \sum_{\alpha = 0}^s \int_{\R} (1+z^2)^n |\partial_z^\alpha f(z)|^2 \dz 
 < + \infty \right\},
\end{equation}
which we endow with their natural norm. We prove in the following lemma that 
vanishing moment conditions yield polynomial decays in these weighted spaces 
for a heat equation set on the real line.

\begin{lemma} \label{lemma.f}
  Let $s, n \in \N$ and $f_0 \in H^{s,n+1}(\R)$ satisfying, 
  for $0 \leq k < n$,
  \begin{equation} \label{eq.f.moment}
    \int_\R z^k f_0(z) \dz  = 0.
  \end{equation}
  Let $f$ \fm{be} the solution to the heat equation on $\R$ with initial data 
  $f_0$:
  \begin{equation} \label{eq.f.after}
    \left\{
    \begin{aligned}
      \partial_t f - \partial_{zz} f & = 0
      &&\quad \textrm{in } \R, \quad \textrm{for } t \geq 0, \\
      f(0,\cdot) & = f_0
      &&\quad \textrm{in } \R, \quad \textrm{for } t = 0.
    \end{aligned}
    \right.
  \end{equation}
  There exists a constant $C_{s,n}$ independent on $f_0$ such that, for 
  $0 \leq m \leq n$,
  \begin{equation} \label{eq.lemma.f}
    \left| f(t,\cdot) \right|_{H^{s,m}}
    \leq 
    C_{s,n}
    \left| f_0 \right|_{H^{s,n+1}} 
    \left| \frac{\ln (2+t)}{2 + t} \right|^{\frac{1}{4} + \frac{n}{2} - 
    \frac{m}{2}}.
  \end{equation}
\end{lemma}

\begin{proof}
For small times (say $t \leq 2$), the \fm{$t$ function in the} right-hand side 
of~\eqref{eq.lemma.f} is
bounded below by a \fm{positive} constant. Thus, inequality~\eqref{eq.lemma.f} 
\fm{holds 
because} the considered energy decays under the heat equation. Let us move on 
to 
large times, e.g. assuming $t \geq 2$. Using Fourier transform in 
$z \mapsto \fourierz$, we compute:
\begin{equation} \label{eq.f.fourier}
  \hat{f}(t,\fourierz) = e^{-t\fourierz^2} \hat{f_0}(\fourierz).
\end{equation}
Moreover, from Plancherel's equality, we have the following estimate:
\begin{equation} \label{eq.f.h2m}
  \left| f(t,\cdot) \right|^2_{H^{s,m}}
  \lesssim
  \sum_{j = 0}^m \int_\R (1+\fourierz^2)^s 
  \left|\partial_\fourierz^j \hat{f}(t, \fourierz)\right|^2 \dif \fourierz.
\end{equation}
We use~\eqref{eq.f.fourier} to compute the derivatives of the Fourier 
transform:
\begin{equation} \label{eq.f.fourier.dj}
  \partial_\fourierz^j \hat{f}(t\fm{,\fourierz})
  = 
  \sum_{i = 0}^j \fourierz^{i-j} P_{i,j}\left(t\fourierz^2\right) 
  e^{-t\fourierz^2} \partial_\fourierz^i \hat{f_0}(\fourierz),
\end{equation}
where $P_{i,j}$ are polynomials with constant numerical coefficients. The 
energy contained at high frequencies decays very fast. For low frequencies,
we will need to use assumptions~\eqref{eq.f.moment}. Writing a Taylor expansion
of $\hat{f}_0$ near $\fourierz = 0$ and taking into account these assumptions 
yields the estimates:
\begin{equation} \label{eq.f.fourier.di.small}
 \left|\partial_\fourierz^i \hat{f}_0(\fourierz)\right| 
 \lesssim |\fourierz|^{n-i} \left|\partial_\fourierz^n\hat{f}_0\right|_{L^\infty}
 \lesssim |\fourierz|^{n-i} \left|z^n f_0(z)\right|_{L^1}
 \lesssim |\fourierz|^{n-i} \left|f_0\right|_{H^{0,n+1}}.
\end{equation}
We introduce $\rho > 0$ and we split the energy integral at a cutting threshold:
\begin{equation} \label{eq.zeta*}
  \fourierz^*(t) := \left|\frac{\rho \ln (2+t)}{2+t}\right|^{1/2}. 
\end{equation}
\textbf{High frequencies.} We start with high frequencies 
$|\fourierz| \geq \fourierz^*(t)$. For large times, this range actually almost includes 
the whole spectrum. Using~\eqref{eq.f.h2m} 
and~\eqref{eq.f.fourier.dj} we compute the high energy terms:
\begin{equation} \label{eq.W.h}
  \mathcal{W}^\sharp_{j,i,i'}(t)
  := \int_{|\fourierz|\geq \fourierz^*(t)}
	\left(1+\fourierz^2\right)^s e^{-2t\fourierz^2} 
	\left|\fourierz\right|^{i-j}
	\left|\fourierz\right|^{i'-j}
	\left|P_{i,j}\left(t\fourierz^2\right)P_{i',j}\left(t\fourierz^2\right)\right|
	\left|\partial_\fourierz^i \hat{f}_0\right|
	\left|\partial_\fourierz^{i'} \hat{f}_0\right| \dif \fourierz.
\end{equation}
Plugging estimate~\eqref{eq.f.fourier.di.small} into~\eqref{eq.W.h} yields:
\begin{equation} \label{eq.W.h2}
 \mathcal{W}^\sharp_{j,i,i'}(t)
 \leq 
 \left|f_0\right|_{H^{0,n+1}}^2 
 \frac{e^{-t(\fourierz^*(t))^2}}{|t|^{n-j+\frac{1}{2}}}
 \int_{\R}
 \left(1+\fourierz^2\right)^s e^{-t\fourierz^2} 
 \left|t\fourierz^{\fm{2}}\right|^{n-j}
 \left|P_{i,j}\left(t\fourierz^2\right)P_{i',j}\left(t\fourierz^2\right)\right|
 t^{\frac{1}{2}} \dif \fourierz.
\end{equation}
The integral in~\eqref{eq.W.h2} is bounded from above for $t \geq 2$ through an
easy change of variable. Moreover, 
\begin{equation} \label{eq.W.h3}
 e^{-t(\fourierz^*(t))^2} = e^{-\frac{\rho t}{2+t} \ln(2+t)}
 = \left(2+t\right)^{-\frac{\rho t}{2+t}}
 \leq \left(2+t\right)^{-\frac{\rho}{2}}.
\end{equation}
Hence, for $t \geq 2$, combining~\eqref{eq.W.h2} and~\eqref{eq.W.h3} yields:
\begin{equation} \label{eq.W.h4}
 \mathcal{W}^\sharp_{j,i,i'}(t)
 \lesssim \left(2+t\right)^{-\frac{\rho}{2}} 
 \left|f_0\right|_{H^{0,n+1}}^2.
\end{equation}
In~\eqref{eq.zeta*}, we can choose any $\rho > 0$. Hence, the decay obtained 
in~\eqref{eq.W.h4} can be arbitrarily good. This is not the case for the low
frequencies estimates which are capped by the number of vanishing moments 
assumed on the initial data $f_0$.

\bigskip
\noindent
\textbf{Low frequencies}. We move on to low frequencies $|\fourierz| \leq 
\fourierz^*(t)$. For large times, this range concentrates near zero. 
Using~\eqref{eq.f.h2m} and~\eqref{eq.f.fourier.dj} we compute the low energy 
terms:
\begin{equation} \label{eq.W.l}
 \mathcal{W}^\flat_{j,i,i'}(t) := \int_{|\fourierz|\leq \fourierz^*(t)}
 \left(1+\fourierz^2\right)^s e^{-2t\fourierz^2} 
 \left|\fourierz\right|^{i-j}
 \left|\fourierz\right|^{i'-j}
 \left|P_{i,j}\left(t\fourierz^2\right)P_{i',j}\left(t\fourierz^2\right)\right|
 \left|\partial_\fourierz^i \hat{f}_0\right|
 \left|\partial_\fourierz^{i'} \hat{f}_0\right| \dif \fourierz.
\end{equation}
Plugging estimate~\eqref{eq.f.fourier.di.small} into~\eqref{eq.W.l} yields:
\begin{equation} \label{eq.W.l2}
 \mathcal{W}^\flat_{j,i,i'}(t)
 \leq 
 \left|f_0\right|_{H^{0,n+1}}^2 
 \int_{|\fourierz|\leq \fourierz^*(t)}
 \left(1+\fourierz^2\right)^s
 \left|\fourierz\right|^{2n-2j}
 \left|P_{i,j}\left(t\fourierz^2\right)P_{i',j}\left(t\fourierz^2\right)\right|
 e^{-2t\fourierz^2}
 \dif \fourierz.
\end{equation}
The function $\fm{\tau \mapsto \left|P_{i,j}(\tau) P_{i',j}(\tau)\right| 
e^{-2\tau}}$ is bounded 
on $\fm{[0,+\infty)}$ thanks to the growth comparison theorem.
\fm{Moreover, $(1+\fourierz^2)^s$ can be bounded by $(1+\rho)^s$ for 
    $|\fourierz|\leq|\fourierz^*(t)|$.}
 Hence, plugging 
the 
definition~\eqref{eq.zeta*} into~\eqref{eq.W.l2} yields:
\begin{equation} \label{eq.W.l3}
 \mathcal{W}^\flat_{j,i,i'}(t)
 \lesssim 
 \left|f_0\right|_{H^{0,n+1}}^2 
 \left|\frac{\rho \ln (2+t)}{2+t}\right|^{\frac{1}{2}+n-j}.
\end{equation}
Hence, choosing $\rho = 1 + 2n - 2m$ in equation~\eqref{eq.zeta*} and summing 
estimates~\eqref{eq.W.h4} with~\eqref{eq.W.l3} for all indexes $0 \leq i, i' 
\leq j \leq m$ concludes the proof of~\eqref{eq.lemma.f} and 
Lemma~\ref{lemma.f}.
\end{proof}

We will use the conclusion of Lemma~\ref{lemma.f} for two different purposes.
First, it states that the boundary layer residue is small at the final time.
Second, estimate~\eqref{eq.lemma.f} can also be used to prove that the source 
terms generated by the boundary layer in the equation of the remainder are 
integrable in large time. Indeed, for $n \geq 2$, $f_0$ and $f$ satisfying the 
assumptions of Lemma~\ref{lemma.f}, we have:
\begin{equation} \label{eq.lemma.f.l1}
 \left\| f \right\|_{L^1(H^{2,n-2})} 
 \lesssim
 \left| f_0 \right|_{H^{2,n+1}}.
\end{equation}

% ==============================================================================
\subsection{Preparation of vanishing moments for the boundary layer profile}

In this paragraph, we explain how we intend to prepare vanishing moments for 
the boundary layer profile at time~$T$ using the control term $\force^v$ of 
equation~\eqref{eq.v}. In order to perform computations within the Fourier 
space in the fast variable, we want to get rid of the Neumann boundary 
condition at $z = 0$. This can be done by lifting the inhomogeneous boundary 
condition $g^0$ to turn it into a source term. We choose the simple lifting 
$-g^0(t,x) e^{-z}$. The homogeneous boundary condition will be preserved via
an even extension of the source term. Let us introduce $V(t,x,z) \in \R^d$ 
defined for $t \geq 0$, $x\in \bar{\Omext}$ and $z \in \R$ \fm{by}:
\begin{equation}
 \fm{V(t,x,z) := v(t,x,|z|) + g^0(t,x) e^{-|z|}}.
\end{equation}
\fm{We also extend implicitly $\force^v$ by parity. Hence, $V$ is the solution
to the following evolution equation:
\begin{equation} \label{eq.V}
 \left\{
 \begin{aligned}
 \partial_t V + (u^0\cdot\nabla)V + BV 
 + u^0_\flat z \partial_z V - \partial_{zz} V 
 & = G^0 e^{-|z|} + \tilde{G}^0 |z| e^{-|z|} + \force^v
 &&\quad \textrm{in } \R_+ \times \bar{\Omext} \times \R_+, \\
 V(0,x,z) & = 0
 &&\quad \textrm{in } \bar{\Omext} \times \R_+,
 \end{aligned}
 \right.
\end{equation}}
where we introduce:
\begin{align}
 \label{eq.def.B}
 B_{i,j} & \fm{:=} \partial_j u^0_i - \left(n \cdot \partial_j u^0\right) n_i
 \fm{+ (u^0 \cdot \nabla n_j) n_i}
  \quad \textrm{for } 1 \leq i, j \leq d, \\
 \label{eq.def.G0} 
 G^0 & \fm{:=} \partial_t g^0 - g^0 + (u^0\cdot\nabla)g^0 + Bg^0, \\
 \label{eq.def.G1}
 \tilde{G}^0 & \fm{:=} - u^0_\flat g^0.
\end{align}
\fm{The null initial condition in~\eqref{eq.V} is due to the fact that 
$u^0(0,\cdot) = 0$ and hence $g^0(0,\cdot) = 0$. Similarly, we have 
$g^0(t,\cdot) = 0$ for $t\geq T$ since we extended $u^0$ by zero after $T$.
    As remarked for equation~\eqref{eq.v}, equation~\eqref{eq.V} also preserves
orthogonality with $n$. Indeed, the particular structure of the zeroth-order 
operator~$B$ 
is such that $\left[(u^0\cdot\nabla)V + BV\right]\cdot n = 0$ for any function 
$V$ such that $V \cdot n = 0$.} 
We compute the partial Fourier transform 
$\hat{V}(t,x,\fourierz) := \int_\R V(t,x,z) e^{-i\fourierz z} \dz$. We obtain:
\begin{equation} \label{eq.vhat.evol}
 \partial_t \hat{V} + (u^0\cdot\fm{\nabla})\hat{V} + 
 \left(B + \fourierz^2 - u^0_\flat\right) \hat{V} 
 \fm{-} u^0_\flat \fourierz \partial_\fourierz \hat{V} 
 = \frac{2G^0}{1+\fourierz^2} + \frac{2 
 \tilde{G}^0(\fm{1-\fourierz^2})}{(1+\fourierz^2)^2}
 + \hat{\force}^v.
\end{equation}
To obtain the decay we are seeking, we will need to consider a finite number of 
derivatives of $\hat{V}$ at $\fourierz = 0$. Thus, we introduce:
\begin{equation} \label{eq.def.q} 
	Q_k(t,x):=\partial_\fourierz^k \hat{V}(t,x,\fourierz = 0).
\end{equation} 
Let us compute the evolution equations satisfied by these quantities. 
Indeed, differentiating equation~\eqref{eq.vhat.evol} $k$ times with 
respect to $\fourierz$ yields:
\begin{equation}
 \begin{split}
 \partial_t \partial_\fourierz^k\hat{V} + (u^0\cdot\nabla) \partial_\fourierz^k 
 \hat{V}+ 
 & \left(B + \fourierz^2 - u^0_\flat\right) \partial_\fourierz^k \hat{V} 
 \fm{+} 2 k\fourierz \partial_\fourierz^{k-1} \hat{V} 
 \fm{+} k(k-1) \partial_\fourierz^{k-2} \hat{V} 
 \fm{-} u^0_\flat \fm{(}\fourierz \partial_\fourierz + k) 
 \partial_\fourierz^{k} 
 \hat{V} \\
 & = \partial^k_\fourierz \left[
 \frac{2G^0}{1+\fourierz^2} + \frac{2 
 \tilde{G}^0(\fm{1-\fourierz^2})}{(1+\fourierz^2)^2}
 + \hat{\force}^v
 \right].
 \end{split}
\end{equation}
Now we can evaluate at $\fourierz = 0$ and obtain:
\begin{equation} \label{eq.qk}
 \partial_t Q_k + (u^0\cdot\nabla)Q_k + \fm{BQ_k}
  \fm{-u^0_\flat (k+1)}  Q_k = 
 \left.
 \partial^k_\fourierz
  \left[
   \frac{2G^0}{1+\fourierz^2} + \frac{2 
   \tilde{G}^0(\fm{1-\fourierz^2})}{(1+\fourierz^2)^2}
   + \hat{\force}^v
  \right]
 \right\rvert_{\fourierz = 0}
 \fm{- k(k-1) Q_{k-2}}.
\end{equation}
In particular:
\begin{align}
 \label{eq.q0}
 \partial_t Q_0 + (u^0\cdot\nabla)Q_0 + \fm{BQ_0 - u^0_\flat Q_0}
 & = 2G^0 \fm{+} 2\tilde{G}^0 + 
 \left[\hat{\force}^v\right]_{\fourierz = 0} \\
 \label{eq.q2}
 \partial_t Q_2 + (u^0\cdot\nabla)Q_2 + \fm{BQ_2 - 3 u^0_\flat Q_2}
 & = \fm{-}2Q_0 - \fm{4}G^0 - \fm{12}\tilde{G}^0 + 
 \left[\partial^2_\fourierz \hat{\force}^v\right]_{\fourierz = 0}.
\end{align}
These equations can be brought back to ODEs using the characteristics method, by
following the flow~$\Phi^0$. Moreover, thanks to their cascade structure, it is
easy to build a source term $\force^v$ which prepares vanishing moments. We have 
the following result:
\begin{lemma} \label{lemma.vanishing}
 Let $n \geq 1$ and $u^0 \in \mathcal{C}^\infty([0,T] \times \bar{\Omext})$
 be a fixed reference flow as defined in paragraph~\ref{par.euler}. There 
 exists $\force^v \in \mathcal{C}^\infty(\fm{\R_+ \times \bar{\Omext} \times 
 \R_+})$ 
 \fm{with $\force^v \cdot n = 0$, whose $x$ support is in 
 $\bar{\Omext}\setminus\bar{\Omega}$, whose time support is compact in
 $(0,T)$,} such that:
 \begin{equation} \label{eq.moments}
  \forall 0 \leq k < n,
  \forall x \in \bar{\Omext},
  \quad
  Q_k(T,x) = 0.
 \end{equation}
 Moreover, for any $s,p \in \N$, for any $0 \leq m \leq n$,
 the associated boundary layer profile satisfies:
 \begin{equation} \label{eq.v.estimate}
   \left|v(t,\cdot,\cdot)\right|_{H^p_x(H^{s,m}_z)}
   \lesssim \left| \frac{\ln(2+t)}{2+t} \right|^{\frac{1}{4}+\frac{n}{2}-
   	\frac{m}{2}},
 \end{equation}
 where the hidden constant depends on the functional space and on $u^0$ but 
 not on the time $t \geq 0$.
\end{lemma}

\begin{proof}
 \textbf{Reduction to independent control of $n$ ODEs.}
 Once $n$ is fixed, let $n' = \lfloor (n-1)/2 \rfloor$. We start by choosing  
 smooth even functions of $z$, $\phi_j$ for $0 \leq j \leq n'$, such that 
 $\partial_\fourierz^{2k} \hat{\phi_j}(0) = \delta_{jk}$. We 
 then compute iteratively the moments $Q_{2j}$ (odd moments automatically 
 vanish by parity) using $\force^v_j := \force^v_j(t,x) \phi_j(z)$ to 
 control $Q_{2j}$ without interfering with previously constructed controls. 
 When computing the control at order $j$, all lower order moments $0 \leq i
 < j$ are known and their contribution as the one of~$Q_0$ in~\eqref{eq.q2}
 can be seen as a known source term.
 
 \textbf{Reduction to a null controllability problem.}
 Let us explain why~\eqref{eq.q0} is controllable. First, by linearity and since
 the source terms $G^0$ and $\tilde{G}^0$ are already known\fm{, fixed and
     tangential}, it 
 suffices to prove that, \fm{starting from zero and without these source terms, 
 we could reach any smooth tangential state.}
 Moreover, since the flow flushing property~\eqref{eq.flow.out} is 
 invariant through time reversal, it is also sufficient to prove that\fm{, in 
 the absence of source term, we can drive any smooth tangential initial state 
 to zero.} These arguments can 
 also be formalized using a Duhamel formula following the flow for 
 equation~\eqref{eq.q0}.
 
 \textbf{Null controllability for a toy system.}
 We are thus left with proving a null controllability property for the following
 toy system:
 \begin{equation} \label{eq.toy.q}
  \left\{
   \begin{aligned}
    \partial_t \fm{Q} + (u^0 \cdot \nabla) \fm{Q} + \fm{B Q + \lambda Q} & = \xi
    && \textrm{in } (0,T) \times \fm{\bar{\Omext}}, \\
    \fm{Q}(0,\cdot) & = \fm{Q}_*
    && \textrm{in } \fm{\bar{\Omext}}, \\
   \end{aligned}
  \right.
 \end{equation}
 \fm{where $\fm{B}(t,x)$ is defined in~\eqref{eq.def.B} and $\lambda(t,x)$ is a 
 smooth scalar-valued} amplification term. Thanks to the
 flushing property~\eqref{eq.flow.out} and to the fact that $\fm{\bar{\Omext}}$
 is bounded, we can choose a finite partition of unity described by functions
 $\eta_l$ for $1 \leq l \leq L$ with $0 \leq \eta_l(x) \leq 1$  and 
 $\sum_l \eta_l \equiv 1$ on $\fm{\bar{\Omext}}$, where the support of $\eta_l$ 
 is a small ball $B_l$ centered at some $x_l \in \fm{\bar{\Omext}}$. Moreover, 
 we extract our partition such that: for any $1 \leq l \leq L$, there exists a
 time $t_l \in (\epsilon,T-\epsilon)$ such that $\mathrm{dist}(\Phi^0(0,t,B_l), 
 \bar{\Omega}) 
 \geq \delta / 2$ for $|t-t_l| \leq \epsilon$ where $\epsilon > 0$. Let 
 $\beta : \R \rightarrow \R$ be a smooth function 
 \fm{with $\beta = 1$ on $(-\infty,-\epsilon)$ and $\beta = 0$ on
 $(\epsilon,+\infty)$}. Let $\fm{Q}^l$ be the solution 
 to~\eqref{eq.toy.q} with initial data $\fm{Q}_*^l := \eta_l \fm{Q}_*$ and null 
 source term $\xi$. We define:
 \begin{align} \label{eq.toy.q2}
  \fm{Q}(t,x) & := \sum_{l = 1}^L \fm{\beta(t-t_l)} \fm{Q}^l(t,x), \\
  \label{eq.toy.q3}
  \fm{\force(t,x)} & \fm{:= \sum_{l = 1}^L \beta'(t-t_l) Q^l(t,x)}.
 \end{align}
 \fm{Thanks to the construction, formulas~\eqref{eq.toy.q2} 
 and~\eqref{eq.toy.q3} define a solution to~\eqref{eq.toy.q} with a smooth 
 control term~$\force$ supported in $\bar{\Omext}\setminus\bar{\Omega}$,
 satisfying $\force \cdot n = 0$ and such that $Q(T,\cdot) = 0$.}
 
 \textbf{Decay estimate.}
 For small times $t \in (0,T)$, when $\force^v \neq 0$, 
 estimate~\eqref{eq.v.estimate} can be seen as a uniform in time estimate and
 can be obtained similarly as the well-posedness results proved 
 in~\cite{MR2754340}. For large times, $t \geq T$, the boundary layer profile
 equation boils down to the parametrized heat equation~\eqref{eq.v.after}
 and we use the conclusion of Lemma~\ref{lemma.f} to deduce~\eqref{eq.v.estimate}
 from~\eqref{eq.lemma.f}.
\end{proof}

% ==============================================================================
\subsection{Staying in a small neighborhood of the boundary}
\label{par.localize}

The boundary layer correction defined in~\eqref{eq.v} is supported within a 
small $x$-neighborhood of~$\BorderExt$. This is legitimate because Navier 
boundary layers don't exhibit separation behaviors. Within this 
$x$-neighborhood, this correction lifts the tangential boundary layer residue 
created by the Euler flow but generates a non vanishing divergence at order
$\sqrt{\varepsilon}$. In the sequel, we will need to find a lifting profile for 
this residual divergence (see~\eqref{eq.def.w}). This will be possible as long 
as the extension $n(x) := - \nabla \varphi(x)$ of the exterior normal to 
$\BorderExt$ does not vanish on the $x$-support of $v$. However, there exists
at least one point in $\Omext$ where $\nabla \varphi = 0$ because $\varphi$
is a non identically vanishing smooth function with $\varphi = 0$ on 
$\BorderExt$. Hence, we must make sure that, despite the transport term present 
in equation~\eqref{eq.v}, the $x$-support of~$v$ will not encounter points where
$\nabla\varphi$ vanishes.

We consider the extended domain~$\Omext$. Its boundary coincides with the
set $\{ x \in \R^d ; \enskip \varphi(x) = 0 \}$. For any $\delta \geq 0$, we
define $\mathcal{V}_\delta := \{ x \in \R^d ; \enskip 0 \leq \varphi(x) \leq 
\delta \}$. Hence, $\mathcal{V}_\delta$ is a neighborhood of~$\BorderExt$
\fm{in $\bar{\Omext}$}. 
For~$\delta$ large enough, $\mathcal{V}_\delta = \fm{\bar{\Omext}}$. As
mentioned in paragraph~\ref{par.def.v}, $\varphi$ was chosen such that 
$|\nabla\varphi| = 1$ and $\fm{|\varphi(x)|} = 
\mathrm{dist}(x,\partial{\Omext})$ 
in a neighborhood of~$\partial{\Omext}$. Let us introduce $\eta > 0$ such 
that this is true on~$\mathcal{V}_\eta$. Hence, within this neighborhood
of~$\partial{\Omext}$, the extension $n(x) = - \nabla \varphi(x)$ of the 
outwards normal to~$\partial{\Omext}$ is well defined (and of unit norm).
We want to guarantee that $v$ vanishes outside of $\mathcal{V}_\eta$.

Considering the evolution equation~\eqref{eq.vhat.evol}, we see it as an 
equation defined on the whole of $\Omext$. Thanks to its structure,
we see that the support of~$\hat{V}$ is transported by the flow of $u^0$.
Moreover, $\hat{V}$ can be triggered either by fixed polluting right-hand side 
source term or by the control forcing term. We want to determine the supports 
of these sources such that $\hat{V}$ vanishes outside of $\mathcal{V}_\eta$.

Thanks to definitions~\eqref{eq.def.g0},~\eqref{eq.def.G0} 
and~\eqref{eq.def.G1}, the unwanted right-hand side source term 
of~\eqref{eq.vhat.evol} is supported within the support of~$\chi$. We
introduce $\eta_\chi$ such that $\supp (\chi) \subset \mathcal{V}_{\eta_\chi}$. 
For $\delta \geq 0$, we define:
\begin{equation} \label{eq.def.S}
 S(\delta) := \sup \left\{ 
 \varphi\left(\Phi^0(t, t', x)\right); 
 \enskip t, t'\in[0,T], 
 \enskip x \in \mathcal{V}_\delta
 \right\} \geq \delta.
\end{equation}
With this notation, $\eta_\chi$ includes the zone where pollution might be 
emitted. Hence $S(\eta_\chi)$ includes the zone that might be reached by some
pollution. Iterating once more, $S(S(\eta_\chi))$ includes the zone where we
might want to act using $\fm{\force^v}$ to prepare vanishing moments. 
Eventually, 
$S(S(S(\eta_\chi))))$ corresponds to the maximum localization of non vanishing
values for $v$.

First, since $u^0$ is smooth, $\Phi^0$ is smooth. Moreover, $\varphi$ is smooth.
Hence,~\eqref{eq.def.S} defines a smooth function of~$\delta$. Second, due to 
the condition $u^0\cdot n= 0$, the characteristics cannot leave or enter the 
domain and thus follow the boundaries. Hence, $S(0) = 0$.  Therefore, by 
continuity of $S$, there exists $\eta_\chi > 0$ small enough such that 
$S(S(S(\eta_\chi)))) \leq \eta$. We assume $\chi$ is fixed from now on.

% ==============================================================================
\subsection{Controlling the boundary layer exactly to zero}

In view of what has been proved in the previous paragraphs, a natural question 
is whether we could have controlled the boundary layer exactly to zero (instead 
of controlling only a finite number of modes and relying on self-dissipation of 
the higher order ones). This was indeed our initial approach but it turned out 
to be impossible. The boundary layer equation~\eqref{eq.v} is not exactly null 
controllable at time $T$. In fact, it is not even exactly null controllable in 
any finite time greater than $T$. Indeed, since $u^0(t,\cdot) = 0$ for $t \geq 
T$, $v$ is the solution to~\eqref{eq.v.after} for $t \geq T$. Hence, reaching 
exactly zero at time~$T$ is equivalent to reaching exactly zero at any later 
time.

Let us present a reduced toy model to explain the difficulty. We consider a 
rectangular domain and a scalar-valued unknown function $v$ solution to the 
following system:
\begin{equation} \label{eq.toy}
 \left\{
 \begin{aligned}
 \partial_t v + \partial_{x} v - \partial_{zz} v & = 0
 && \quad [0,T] \times [0,1] \times [0,1], \\
 v(t,x,0) & = g(t,x) 
 && \quad [0,T] \times [0,1], \\
 v(t,x,1) & = 0 
 && \quad [0,T] \times [0,1], \\
 v(t,0,z) & = q(t,z) 
 && \quad [0,T] \times [0,1], \\
 v(0,x,z) & = 0
 && \quad [0,1] \times [0,1].
 \end{aligned}
 \right.
\end{equation}
System~\eqref{eq.toy} involves both a known tangential transport term and a
normal diffusive term. At the bottom boundary, $g(t,x)$ is a smooth fixed 
pollution source term (which models the action of $N(u^0)$, the boundary layer
residue created by our reference Euler flow). At the left inlet vertical 
boundary $x = 0$, we can choose a Dirichlet boundary value control $q(t,z)$. 
Hence, applying the same strategy as described above, we can control any finite
number of vertical modes \fm{provided that $T \geq 1$.}

However, let us check that it would not be reasonable to try to control the 
system exactly to zero at \fm{any given time $T \geq 1$}. Let us consider a 
vertical slice located at $x_\star \in (0,1)$ of the domain at the final time 
and follow the flow backwards by defining:
\begin{equation} \label{eq.def.vstar}
 v_\star(t,z) := v(t,x_\star+(t-T),z).
\end{equation}
Hence, letting $T_\star := T - x_\star \fm{\geq 0}$ and using~\eqref{eq.def.vstar}, 
$v_\star$ is the solution to a one dimensional heat system:
\begin{equation} \label{eq.toy.1d}
\left\{
\begin{aligned}
\partial_t v_\star - \partial_{zz} v_\star & = 0
&& \quad [T_\star,T] \times [0,1], \\
v_\star(t,0) & = g_\star(t) 
&& \quad [T_\star,T], \\
v_\star(t,1) & = 0 
&& \quad [T_\star,T], \\
v_\star(0,z) & = q_\star(z)
&& \quad [0,1],
\end{aligned}
\right.
\end{equation}
where $g_\star(t) := g(t,x_\star+(t-T))$ is smooth but fixed and 
$q_\star(z) := q(T_\star,z)$ is an initial data that we can choose as if it was 
a control. Actually, let us change a little the definition of $v_\star$ to 
lift the inhomogeneous boundary condition at $z = 0$. We set:
\begin{equation} \label{eq.def.vstar2}
 v_\star(t,z) := v(t,x_\star+(t-T),z) - (1-z)g_\star(t).
\end{equation}
Hence, system~\eqref{eq.toy.1d} \fm{reduces} to:
\begin{equation} \label{eq.toy.1d.2}
 \left\{
 \begin{aligned}
 \partial_t v_\star - \partial_{zz} v_\star & = -(1-z)g_\star'(t)
 && \quad [T_\star,T] \times [0,1], \\
 v_\star(t,0) & = 0
 && \quad [T_\star,T], \\
 v_\star(t,1) & = 0 
 && \quad [T_\star,T], \\
 v_\star(0,z) & = q_\star(z) 
 && \quad [0,1],
 \end{aligned}
 \right.
\end{equation}
where we change the definition of $q_\star(z) := q(T_\star,z) - (1-z) 
g_\star(T_\star)$. Introducing the Fourier basis adapted to 
system~\eqref{eq.toy.1d.2}, $e_n(z) := \sin(n\pi z)$, we can solve explicitly 
for the evolution of $v_\star$:
\begin{equation} \label{eq.vstar.fourier}
 v_\star^n(T) = e^{-n^2\pi^2 T} v_\star^n(0)
 - \int_{T_\star}^T e^{-n^2 \pi^2 (T-t)} \langle 1 - z, e_n \rangle g_\star'(t) 
 \dt.
\end{equation}
If we assume that the pollution term $g$ vanishes at the final time, 
equation~\eqref{eq.vstar.fourier} and exact null controllability  would impose 
the choice of the initial control data:
\begin{equation} \label{eq.vstar.fourier2}
 q_\star^n = \langle 1 - z, e_n \rangle \int_{T_\star}^T e^{n^2 \pi^2 t} 
 g_\star'(t) \fm{\dt}.
\end{equation}
Even if the pollution term $g$ is very smooth, there is nothing good to be 
expected from relation~\eqref{eq.vstar.fourier2}. Hoping for cancellations or
vanishing moments is not reasonable because we would have to guarantee this 
relation for all Fourier modes $n$ and all $x_\star \in [0,1]$. Thus, the 
boundary data control that we must choose has exponentially growing Fourier
modes. \fm{Heuristically, it belongs to the dual of a Gevrey space.}

The intuition behind relation~\eqref{eq.vstar.fourier2} is that the control
data emitted from the left inlet boundary \fm{undergoes} a heat regularization 
process as they move towards their final position. In the meantime, 
\fm{the} fixed 
polluting boundary data is injected directly at positions within the domain and
undergoes less smoothing. This prevents any hope from proving exact null 
controllability for system~\eqref{eq.toy} within reasonable functional spaces
and explains why we had to resort to a low-modes control process.

Theorem~\ref{thm.weak.null} is an exact null controllability result. To 
conclude our proof, we use a local argument stated as Lemma~\ref{lemma.local} in 
paragraph~\ref{par.proof} which uses diffusion in all directions. Boundary 
layer systems like~\eqref{eq.toy} exhibit no diffusion in the tangential 
direction and are thus harder to handle. The conclusion of our proof uses 
the initial formulation of the Navier-Stokes equation with a fixed 
$\mathcal{O}(1)$ viscosity.

% ==============================================================================
\section{Estimation of the remainder and technical profiles}
\label{section.remainder}
% ==============================================================================

In the previous sections, we presented the construction of the Euler 
reference flushing trajectory~$u^0$, the transported flow involving the 
initial data~$u^1$ and the leading order boundary layer correction~$v$.
In this section, we follow on with the expansion and introduce technical 
profiles, which do not have a clear physical interpretation. The purpose of the 
technical decomposition we propose is to help us prove that the remainder we 
obtain is indeed small. We will use the following expansion:
\begin{align} 
\label{eq.expansion.full}
\ue & = u^0 
+ \sqrt{\varepsilon} \eval{v} 
+ \varepsilon u^1
+ \varepsilon \nabla \theta^\varepsilon
+ \varepsilon \eval{w}
+ \varepsilon \fm{\Rem}, \\
\label{eq.expansion.p.full}
p^\varepsilon 
& = p^0
+ \varepsilon \eval{q}
+ \varepsilon p^1 
+ \varepsilon \mu^\varepsilon
+ \varepsilon \pi^\varepsilon,
\end{align}
where $v$, $w$ and $q$ are profiles depending on $t,x$ and $z$. For such a 
function $f(t,x,z)$, we use the notation $\eval{f}$ to denote its 
evaluation at $z = \varphi(x)/\sqrt{\varepsilon}$. In the sequel, operators 
$\nabla$, $\Delta$, $D$ and $\diverg$ only act on $x$ variables. We will use 
the following straightforward commutation formulas:
\begin{align}
\label{eq.eval.div}
\diverg \eval{f} 
& = \eval{\diverg f} - n \cdot \eval{\partial_z f} / \sqrt{\varepsilon} \\
\label{eq.eval.nabla}
\nabla \eval{f} 
& = \eval{\nabla f} - n \eval{\partial_z f}  / 
\sqrt{\varepsilon}, \\
\label{eq.eval.navier}
N(\eval{f}) 
& = \eval{N(f)} - \frac{1}{2} \eval{\tanpart{\partial_z f}}  / 
\sqrt{\varepsilon}, \\
\label{eq.eval.delta}
\varepsilon \Delta \eval{f}
& = \varepsilon \eval{\Delta f} + \sqrt{\varepsilon} \Delta \varphi 
\eval{\partial_z f} - 2 \sqrt{\varepsilon} \eval{(n \cdot \nabla) \partial_z f}
+ |n|^2 \eval{\partial_{zz} f}.
\end{align}
Within the $x$-support of boundary layer terms, $|n|^2$ = 1.

% ==============================================================================
\subsection{Formal expansions of constraints}

In this paragraph, we are interested in the formulation of the boundary 
conditions and the incompressibility condition for the full expansion. 
We plug expansion~\eqref{eq.expansion.full} into these conditions and 
identify the successive orders of power of $\sqrt{\varepsilon}$.

\subsubsection{Impermeability boundary condition}

The impermeability boundary condition $\ue \cdot n = 0$ on~$\BorderExt$ yields:
\begin{align}
 \label{eq.imper.u0}
 u^0 \cdot n & = 0, \\
 \label{eq.imper.v}
 v(\cdot,\cdot,0) \cdot n & = 0, \\
 \label{eq.imper.1}
 u^1 \cdot n + \partial_n \theta^\varepsilon
 + w(\cdot,\cdot,0) \cdot n + \Rem \cdot n & = 0.
\end{align}
By construction of the Euler trajectory $u^0$, equation~\eqref{eq.imper.u0} is 
satisfied. Since the boundary profile $v$ is tangential, 
equation~\eqref{eq.imper.v} is also satisfied. By construction, we also already 
have $u^1\cdot n = 0$. In order to be able to carry out integrations by part 
for the estimates of the remainder, we also would like to impose $\Rem \cdot n 
= 0$. Thus, we read~\eqref{eq.imper.1} as a definition of~$\partial_n 
\theta^\varepsilon$ once $w$ is known:
\begin{equation} \label{eq.flux.u1.w}
 \forall t \geq 0, \forall x \in \BorderExt, \quad 
 \partial_n \theta^\varepsilon(t,x) = - w(t,x,0) \cdot n. 
\end{equation}

\subsubsection{Incompressibility condition}

The (almost) incompressibility condition $\diverg \ue = \sigma^0$ in $\Omext$ 
($\sigma^0$ is smooth forcing terms supported outside of the physical domain $\Omega$) yields:
\begin{align}
 \label{eq.diverg.0}
 \diverg u^0 - n \cdot \eval{\partial_z v}& = 
 \sigma^0, \\
 \label{eq.diverg.12}
 \eval{\diverg v} - n \cdot \eval{\partial_z w}& = 0, \\
 \label{eq.diverg.1}
 \diverg u^1 + \diverg \nabla \theta^\varepsilon 
 + \eval{\diverg w} + \diverg \Rem & = 0.
\end{align}
In~\eqref{eq.diverg.12} and~\eqref{eq.diverg.1}, we used 
formula~\eqref{eq.eval.div} to isolate the contributions to the
divergence coming from the slow derivatives with the one coming from the fast
derivative $\partial_z$. By construction $\diverg u^0 = \sigma^0$, 
$\diverg u^1 = 0$, $n \cdot \partial_z v = 0$ and we would like to work 
with $\diverg \Rem = 0$. Hence, we read~\eqref{eq.diverg.12} 
and~\eqref{eq.diverg.1} as:
\begin{align}
 \label{eq.diverg.w}
 n \cdot \eval{\partial_z w} & = \eval{\diverg v}, \\
 \label{eq.diverg.u2}
 - \Delta \theta^\varepsilon & = \eval{\diverg w}.
\end{align}

\subsubsection{Navier boundary condition}

Last, we turn to the slip-with-friction boundary condition. Proceeding as above 
yields by identification:
\begin{align}
 \label{eq.navier.0}
 N(u^0) - \frac{1}{2} \tanpart{\partial_z v} \big\vert_{z=0} & = 0, \\
 \label{eq.navier.12}
 N(v) \big\vert_{z=0}
 - \frac{1}{2} \tanpart{\partial_z w} \big\vert_{z=0} & = 0, \\
 \label{eq.navier.1}
 N(u^1) + N(\nabla \theta^\varepsilon) + N(w)\big\vert_{z=0} + N(\Rem) & = 0.
\end{align}
By construction,~\eqref{eq.navier.0} is satisfied. We will choose a basic 
lifting to guarantee~\eqref{eq.navier.12}. Last, we read~\eqref{eq.navier.1} as 
\fm{an} inhomogeneous boundary condition for the remainder:
\begin{equation} \label{eq.def.g}
 N(\Rem) = g^\varepsilon := - N(u^1) - N(\nabla \theta^\varepsilon) 
 - N(w) \big\vert_{z=0}. 
\end{equation}

% ==============================================================================
\subsection{Definitions of technical profiles}

At this stage, the three main terms $u^0$, $v$ and $u^1$ are defined. In this 
paragraph, we explain step by step how we build the following technical 
profiles of the expansion. For any $t \geq 0$, the profiles are built 
sequentially from the values of~$v(t,\cdot,\cdot)$. Hence, they will inherit
from the boundary layer profile its smoothness with respect to the slow 
variables $x$ and its time decay estimates obtained from Lemma~\ref{lemma.f}.

\subsubsection{Boundary layer pressure}

Equation~\eqref{eq.v} only involves the tangential part of the symmetrical 
convective product between~$u^0$ and~$v$. Hence, to compensate its normal part, 
we introduce as in~\cite{MR2754340} the pressure $q$ which is defined as the 
unique solution vanishing as $z\rightarrow +\infty$ to:
\begin{equation} \label{eq.q}
 \left[ (u^0 \cdot \nabla) v + (v \cdot \nabla) u^0 \right] \cdot n
 = \partial_z q.
\end{equation}
Hence, we can now write:
\begin{equation} \label{eq.v.full}
 \partial_t v + (u^0 \cdot \nabla) v + (v \cdot \nabla) u^0
   + u^0_\flat z \partial_z v - \partial_{zz} v - n \partial_z q = 0.
\end{equation}
This pressure profile vanishes as soon as $u^0$ vanishes, hence in 
particular for $t \geq T$. For any $p, s, n \in \N$, the following estimate
is straightforward:
\begin{equation} \label{eq.estimate.qpressure}
 \fm{
 \left| q(t, \cdot, \cdot) \right|_{H^1_x(H^{0,0}_z)}
 \lesssim 
 \left| v(t, \cdot, \cdot) \right|_{H^2_x(H^{0,2}_z)}
 }.
\end{equation}

\subsubsection{Second boundary corrector}

The first boundary condition~$v$ generates a non vanishing slow divergence and 
a non vanishing tangential boundary flux. The role of the profile $w$ is to 
lift two unwanted terms that would be too hard to handle directly in the 
equation of the remainder. We define $w$ as:
\begin{align} \label{eq.def.w}
 w(t,x,z) := - 2 e^{-z} N(v)(t,x,0) - n(x) \int_z^{+\infty} \diverg 
 v(t,x,z')\dz'
\end{align}
Definition~\eqref{eq.def.w} allows to guarantee condition~\eqref{eq.navier.12}. 
Moreover, under the assumption $|n(x)|^2 = 1$ for any~$x$ in the $x$-support of 
the boundary layer, this definition also fulfills  
condition~\eqref{eq.diverg.12}. In equation~\eqref{eq.def.w} it is essential
that $n(x)$ does not vanish on the $x$-support of $v$. This is why we dedicated
paragraph~\ref{par.localize} to proving we could maintain a small enough 
support for the boundary layer. For any $p, s, n \in \N$, the following 
estimates are straightforward:
\begin{align}
\label{eq.estimate.w.1}
\left| \tanpart{w(t, \cdot, \cdot)} \right|_{H^p_x(H^{s,n}_z)}
& \lesssim 
\left| v(t, \cdot, \cdot) \right|_{H^{p+1}_x(H^{1,1}_z)},  \\
\label{eq.estimate.w.2}
\left| w(t, \cdot, \cdot) \cdot n \right|_{H^p_x(H^{0,n}_z)}
& \lesssim 
\left| v(t, \cdot, \cdot) \right|_{H^{p+1}_x(H^{0,n+2}_z)},  \\
\label{eq.estimate.w.3}
\left| w(t, \cdot, \cdot) \cdot n \right|_{H^p_x(H^{s+1,n}_z)}
& \lesssim 
\left| v(t, \cdot, \cdot) \right|_{H^{p+1}_x(H^{s,n}_z)}.
\end{align}
Estimates~\eqref{eq.estimate.w.1},~\eqref{eq.estimate.w.2} 
and~\eqref{eq.estimate.w.3} can be grossly summarized sub-optimally by:
\begin{equation} \label{eq.estimate.w.general}
\left| w(t, \cdot, \cdot) \right|_{H^p_x(H^{s,n}_z)}
\lesssim 
\left| v(t, \cdot, \cdot) \right|_{H^{p+1}_x(H^{s+1,n+2}_z)}.
\end{equation}

\subsubsection{Inner domain corrector}

Once $w$ is defined by~\eqref{eq.def.w}, the collateral damage is that this 
generates a non vanishing boundary flux $w \cdot n$ on $\fm{\BorderExt}$ and a
slow divergence. For a fixed time $t \geq 0$, we define $\theta^\varepsilon$
as the solution to:
\begin{equation} \label{eq.theta1}
 \left\{
 \begin{aligned}
    \Delta \theta^\varepsilon & = - \eval{\diverg w} && 
    \quad \textrm{in }\Omext, \\
    \partial_n \theta^\varepsilon & = -w(t,\cdot,0)\cdot n && 
    \quad \textrm{\fm{on} }\BorderExt.
 \end{aligned}
 \right.
\end{equation}
System~\eqref{eq.theta1} is well-posed as soon as the usual compatibility 
condition between the source terms is satisfied. Using Stokes formula,
equations~\eqref{eq.eval.div} and~\eqref{eq.diverg.12}, we compute:
\begin{equation}
 \begin{split}
  \int_{\BorderExt} w(t,\cdot,0) \cdot n
  = \int_{\BorderExt} \eval{w} \cdot n
  & = \int_{\Omext} \diverg \eval{w} \\
  & = \int_{\Omext} \eval{\diverg w} - \varepsilon^{-\frac{1}{2}}
  n \cdot \eval{\partial_z w} \\
  & = \int_{\Omext} \eval{\diverg w} - \varepsilon^{-\frac{1}{2}}
  \eval{\diverg v} \\
  & = \int_{\Omext} \eval{\diverg w} - \varepsilon^{-\frac{1}{2}}
  \diverg \eval{v} + \varepsilon^{-1} n \cdot 
  \eval{\partial_z v} \\
  & = \int_{\Omext} \eval{\diverg w} - 
  \varepsilon^{-\frac{1}{2}} \int_{\BorderExt} \eval{v} \cdot n
  = \int_{\Omext} \eval{\diverg w},
 \end{split}
\end{equation}
where we used twice the fact that $v$ is tangential. Thus, the compatibility
condition is satisfied and system~\eqref{eq.theta1} has a unique solution.
The associated potential flow $\nabla\theta^\varepsilon$ solves:
\begin{equation} \label{eq.tilde.u1}
	\left\{
	\begin{aligned}
		\partial_t \nabla\theta^\varepsilon
		+ \left( u^0 \cdot \nabla \right) \nabla\theta^\varepsilon
		+ \left( \nabla\theta^\varepsilon \cdot \nabla \right) u^0
		+ \nabla \mu^\varepsilon
		& = 0,
		&&\quad \textrm{in } \Omext \quad \textrm{for } t \geq 0, \\
		\diverg \nabla\theta^\varepsilon & = -\eval{\diverg w}
		&&\quad \textrm{in } \Omext \quad \textrm{for } t \geq 0, \\
		\nabla\theta^\varepsilon \cdot n & = - w_{\rvert z = 0} \cdot n
		&&\quad \textrm{\fm{on} } \BorderExt \quad \textrm{for } t \geq 0,
	\end{aligned}
	\right.
\end{equation}
where the pressure term  
$\mu^\varepsilon := - \partial_t \theta^\varepsilon
- u^0 \cdot \nabla\theta^\varepsilon$ absorbs all \fm{other} terms in the
evolution equation \fm{(see~\eqref{eq.euler.irrot})}. Estimating roughly 
$\theta^\varepsilon$ using standard
regularity estimates for the Laplace equation yields:
\begin{equation} \label{eq.estimate.theta.h4}
 \begin{split}
  \left| \theta^\varepsilon(t,\cdot) \right|_{H^4_x}
  & \lesssim
  \left| \eval{\diverg w}(t,\cdot) \right|_{H^2_x} +
  \left| w(t,\cdot,0) \cdot n \right|_{H^3_x} \\
  & \lesssim
  \varepsilon^{\frac{1}{4}}
  \left| w(t) \right|_{H^4_x(H^{0,0}_z)} +
  \varepsilon^{-\frac{1}{4}}
  \left| w(t) \right|_{H^3_x(H^{1,0}_z)} +
  \varepsilon^{-\frac{3}{4}}
  \left| w(t) \right|_{H^2_x(H^{2,0}_z)} +
  \left| v(t) \right|_{H^3_x(H^{0,1}_z)} \\
  & \lesssim \varepsilon^{-\frac{3}{4}} 
  \left| w(t) \right|_{H^4_x(H^{2,0}_z)}
  + \left| v(t) \right|_{H^3_x(H^{0,1}_z)},
 \end{split}
\end{equation}
where we used~\cite[Lemma 3, page 150]{MR2214949} to benefit from the fast 
variable scaling. Similarly,
\begin{align} \label{eq.estimate.theta.h3}
 \left| \theta^\varepsilon(t,\cdot) \right|_{H^3_x}
 & \lesssim \varepsilon^{-\frac{1}{4}} 
 \left| w(t) \right|_{H^3_x(H^{1,0}_z)} 
 + \left| v(t) \right|_{H^2_x(H^{0,1}_z)}, \\
 \label{eq.estimate.theta.h2}
 \left| \theta^\varepsilon(t,\cdot) \right|_{H^2_x}
 & \lesssim \varepsilon^{\frac{1}{4}} \left| w(t) \right|_{H^2_x(H^{0,0}_z)}
 + \left| v(t) \right|_{H^1_x(H^{0,1}_z)}.
\end{align}

% ==============================================================================
\subsection{Equation for the remainder}

In the extended domain $\Omext$, the remainder is a solution to:
\begin{equation} \label{eq.R}
 \left\{
 \begin{aligned}
   \partial_t \Rem - \varepsilon \Delta \Rem
   + \left( \ue \cdot \nabla \right) \Rem
   + \nabla \pi^\varepsilon 
   & = \eval{f^\varepsilon} - \eval{A^\varepsilon \Rem}
   &&\quad \textrm{in } \Omext \quad \textrm{for } t \geq 0, \\
   \diverg \Rem & = 0
   &&\quad \textrm{in } \Omext \quad \textrm{for } t \geq 0, \\
   N(\Rem)
   & = g^\varepsilon
   &&\quad \textrm{\fm{on} } \BorderExt \quad \textrm{for } t \geq 0, \\
   \Rem \cdot  n& = 0 
   &&\quad \textrm{\fm{on} } \BorderExt \quad \textrm{for } t \geq 0, \\
   \Rem(0, \cdot) & = 0
   &&\quad \textrm{in } \Omext \quad \textrm{at } t = 0.
 \end{aligned}
 \right.
\end{equation}
\fm{Recall that $g^\varepsilon$ is defined in~\eqref{eq.def.g}.}
We introduce the amplification operator:
\begin{equation} \label{eq.def.A}
 A^\varepsilon \Rem 
 := 
 (\Rem \cdot \nabla) \left( u^0 + \sqrt{\varepsilon}v + \varepsilon u^1 + 
 \varepsilon \nabla \theta^\varepsilon + \varepsilon w \right)
 - (\Rem \cdot n) \left( \partial_z v + \sqrt{\varepsilon} \partial_z 
 w \right)
\end{equation}
and the forcing term:
\begin{equation} \label{eq.def.f}
 \begin{split}
  f^\varepsilon 
  := &
  ( \Delta \varphi \partial_z v
    - 2 \fm{(n\cdot\nabla)} \partial_{z} v 
    + \partial_{zz} w )
  + \sqrt{\varepsilon} ( \Delta v
    + \Delta \varphi \partial_z w
    - 2 \fm{(n\cdot\nabla)} \partial_{z} w )
  + \varepsilon ( \Delta w + \Delta u^1 
	+ \Delta \nabla \theta^\varepsilon ) \\
  & 
  - \left((v + \sqrt{\varepsilon}(w+u^1+\nabla\theta^\varepsilon)) \fm{\cdot} 
  \nabla\right)
		(v + \sqrt{\varepsilon}(w+u^1+\nabla\theta^\varepsilon))
		- \fm{(u^0 \cdot \nabla)} w - \fm{(w\cdot \nabla)} u^0  \\
 &
		- u^0_\flat z \partial_z w
		+ (w+u^1 + \nabla\theta^\varepsilon)\cdot n\partial_z 
		\left(v + \sqrt{\varepsilon}w\right)
		\\
  & - \nabla q - \partial_t w.
 \end{split}
\end{equation}
In~\eqref{eq.def.A} and~\eqref{eq.def.f}, many functions depend on $t, x$ and
$z$. The differential operators $\nabla$ and $\Delta$ only act on the slow 
variables $x$ and the evaluation at $z = \varphi(x) / \sqrt{\varepsilon}$ is 
done \emph{a posteriori} in~\eqref{eq.R}. The derivatives in the fast variable 
direction are explicitly marked with the~$\partial_z$ operator. Moreover, most 
terms are independent of $\varepsilon$, except where explicitly stated in 
$\theta^\varepsilon$ and~$\Rem$.

Expansion~\eqref{eq.expansion.full} contains 4 slowly varying profiles and 2
boundary layer profiles. Thus, computing $\varepsilon \Delta \ue$ using 
formula~\eqref{eq.eval.delta} produces $4 + 2 \times 4 = 12$ terms. Terms 
$\Delta u^0$ and $\eval{\partial_{zz} v}$ have already been taken into account
respectively in~\eqref{eq.u1.ext} and~\eqref{eq.v}. Term $\Delta\Rem$ is written
in~\eqref{eq.R}. The remaining $9$ terms are gathered in the first line of
the forcing term~\eqref{eq.def.f}.

Computing the non-linear term $\fm{(\ue\cdot\nabla)}\ue$ using 
formula~\eqref{eq.eval.nabla} produces $6 \times 4 + 6 \times 2 \times 2 = 48$
terms. First, $8$ have already been taken into account in~\eqref{eq.euler.ext},
\eqref{eq.u1.ext},~\eqref{eq.v} and~\eqref{eq.tilde.u1}. Moreover, $6$ are written 
in~\eqref{eq.R} as $\fm{(\ue\cdot\nabla)}\Rem$, $7$ more as the amplification 
term
\eqref{eq.def.A} and $25$ in the second and third line of~\eqref{eq.def.f}.
The two missing terms $\eval{(v \cdot n)\partial_z v}$ and $\eval{(v \cdot 
n)\partial_z w}$ vanish because $v \cdot n = 0$.

% ==============================================================================
\subsection{Size of the remainder}

We need to prove that equation~\eqref{eq.R} satisfies an energy estimate on the
long time interval $[0,T/\varepsilon]$. Moreover, we need to estimate the size 
of the remainder at the final time and check that it is small. The key point is 
that the size of the source term $\eval{f^\varepsilon}$  is small in 
$L^2(\Omext)$. Indeed, for terms appearing at order $\mathcal{O}(1)$, the 
fast scaling makes us win a $\varepsilon^{\frac{1}{4}}$ factor (see 
for example~\cite[Lemma~3, page~150]{MR2754340}). We proceed as we have done in 
the case of the shape operator in paragraph~\ref{par.r.estimate}. 

The only difference is the estimation of the boundary 
term~\eqref{eq.r.estimate.bt.shape}. We have to take into account the 
inhomogeneous boundary condition $g^\varepsilon$ and the fact that, in the 
general case, the boundary condition matrix $M$ is different from the shape 
operator $\MW$. Using~\eqref{eq.m.to.rot} allows us to write, on $\BorderExt$:
\begin{equation} \label{eq.m.to.mw}
 \fm{\left(\Rem \times (\nabla \times \Rem)\right)} \cdot n 
 = \fm{\left((\nabla \times \Rem) \times n\right)} \cdot\Rem
 = 2 \left(N(\Rem) + \tanpart{(M-\MW) \Rem } \right) \cdot \Rem.
\end{equation}
Introducing smooth extensions of $M$ and $\MW$ to the whole
domain $\Omext$ also \fm{allows to extend} the Navier operator~$N$ defined 
in~\eqref{eq.def.navier}\fm{, since the extension of the normal $n$ extends 
the definition of the tangential part~\eqref{eq.def.tanpart}. 
Using~\eqref{eq.m.to.mw}}, we transform the boundary term into an inner term:
\begin{equation} \label{eq.rr.boundary}
 \begin{split}
 \left| \int_{\BorderExt} \fm{\left(\Rem \times (\nabla \times \Rem)\right)} 
 \cdot n  \right| 
 & = 2 \left| \int_{\BorderExt} g^\varepsilon \fm{\cdot} \Rem + ((M-\MW) 
 \Rem)\fm{\cdot}\Rem
 \right| \\
 & = 2 \left| \int_{\Omext} \diverg \left[ (g^\varepsilon \fm{\cdot}\Rem) n + 
 (((M-\MW) 
 \Rem)\fm{\cdot}\Rem) n \right] \right| \\
 & \fm{\leq \lambda \left|\nabla\Rem\right|_2^2
 + C_\lambda \left( \left|\Rem\right|_2^2 + 
 + \left|g^\varepsilon\right|_2^2 
 + \left|\nabla g^\varepsilon\right|_2^2 \right),}
 \end{split}
\end{equation}
\fm{for any $\lambda > 0$ to be chosen and where $C_\lambda$ is a positive 
constant depending on $\lambda$. We intend to absorb the~$|\nabla\Rem|_2^2$ 
term of~\eqref{eq.rr.boundary} using the dissipative term. However, the 
dissipative term only provides the norm of the symmetric part of the gradient. 
We recover the full gradient using the Korn inequality. Indeed, since $\diverg 
\Rem = 0$ in $\Omext$ and $\Rem \cdot n = 0$ on $\partial \Omext$, the 
following estimate holds (see~\cite[Corollary~1, Chapter~IX, 
page~212]{MR1064315}):
\begin{equation} \label{eq.korn}
\left|\Rem\right|_{H^1(\Omext)}^2 
\fm{\leq} \fm{C_K} \left|\Rem\right|_{L^2(\Omext)}^2 + \fm{C_K} 
\left|\nabla\times\Rem\right|_{L^2(\Omext)}^2.
\end{equation}
We choose $\lambda = 1/(2C_K)$ in~\eqref{eq.rr.boundary}. Combined 
with~\eqref{eq.korn}
and a Grönwall} inequality as in 
paragraph~\ref{par.r.estimate} yields an 
energy estimate for $t \in [0,T/\varepsilon]$:
\begin{equation} \label{eq.R.gron3}
 \left|\Rem\right|^2_{L^\infty(L^2)} + 
 \varepsilon \left|\Rem\right|^2_{L^2(H^1)}
 = \mathcal{O}(\varepsilon^{\frac{1}{4}}),
\end{equation}
as long as we can check that the following estimates hold:
\begin{align}
 \label{eq.estimate.a}
 \left\| A^\varepsilon \right\|_{L^1(L^\infty)} & = \mathcal{O}(1), \\
 \label{eq.estimate.g}
 \varepsilon \left\| g^\varepsilon \right\|^2_{L^2(H^1)} & = 
 \mathcal{O}(\varepsilon^\frac{1}{4}), \\
 \label{eq.estimate.f}
 \left\| f^\varepsilon \right\|_{L^1(L^2)} & = 
 \mathcal{O}(\varepsilon^\frac{1}{4}).
\end{align}
In particular, the remainder at time $T/{\varepsilon}$ is small and we can 
conclude the proof of Theorem~\ref{thm.weak.null} with the same arguments as
in paragraph~\ref{par.proof}. Therefore, it only remains to be checked that 
estimates~\eqref{eq.estimate.a},~\eqref{eq.estimate.f} and~\eqref{eq.estimate.g}
hold on the time interval $[0, T/\varepsilon]$. In fact, they even hold on 
the whole time interval $[0,+\infty)$.

\bigskip \noindent \textbf{Estimates for $A^\varepsilon$.} The two terms 
involving $u^0$ and $u^1$ vanish for $t \geq T$. Thus, they satisfy 
estimate~\eqref{eq.estimate.a}. For $t \geq 0$, we estimate the other terms
in $A^\varepsilon$ in the following way:
\begin{align}
 \sqrt{\varepsilon} \left| \nabla v(t) \right|_{L^\infty}
 & \lesssim \sqrt{\varepsilon} \left| v(t) \right|_{H^3_x(H^{1,\fm{0}}_z)}, \\
 {\varepsilon} \left| \nabla w(t) \right|_{L^\infty}
 & \lesssim {\varepsilon} \left| w(t) \right|_{H^3_x(H^{1,\fm{0}}_z)}, \\
 \left| \partial_z v(t) \right|_{L^\infty}
 & \lesssim \left| v(t) \right|_{H^2_x(H^{2,\fm{0}}_z)}, \\
 \sqrt{\varepsilon} \left| \partial_z w(t) \right|_{L^\infty}
 & \lesssim \sqrt{\varepsilon} \left| w(t) \right|_{H^2_x(H^{2,\fm{0}}_z)}, \\
 {\varepsilon} \left| \nabla^2 \theta^\varepsilon(t) \right|_{L^\infty}
 & \lesssim \varepsilon \left| \theta^\varepsilon(t) \right|_{H^4}.
\end{align}
Combining these estimates with~\eqref{eq.estimate.theta.h4} 
and~\eqref{eq.estimate.w.general} yields:
\begin{equation}
 \left\| A^\varepsilon \right\|_{L^1(L^\infty)}
 \lesssim \| u^0 \|_{L^1_{[0,T]}(H^3)}
 + \varepsilon \| u^1 \|_{L^1_{[0,T]}(H^3)}
 + \left\| v \right\|_{L^1(H^5_x(H^{3,\fm{2}}_z))}.
\end{equation}
Applying Lemma~\ref{lemma.vanishing} with $p = 5$, $n = \fm{4}$ and $m = 
\fm{2}$ concludes the proof of~\eqref{eq.estimate.a}.

\bigskip \noindent \textbf{Estimates for $g^\varepsilon$.} For $t \geq 0$, 
using the definition of $g^\varepsilon$ in~\eqref{eq.def.g}, we estimate:
\begin{align}
 \varepsilon \left|N(u^1)(t)\right|_{H^1}^2
 & \lesssim \varepsilon \left|u^1(t)\right|_{H^2}^2, \\
 \varepsilon \left|N(\nabla\theta^\varepsilon)(t)\right|_{H^1}^2
 & \lesssim \varepsilon \left|\theta^\varepsilon(t)\right|_{H^3}^2, \\
 \varepsilon \left|N(w)_{\rvert z = 0}(t)\right|_{H^1}^2
 & \lesssim \varepsilon \left|w(t)\right|_{H^2_x(H^{1,1}_z)}^2.
\end{align}
Combining these estimates with~\eqref{eq.estimate.theta.h3} 
and~\eqref{eq.estimate.w.general} yields:
\begin{equation}
 \varepsilon \left\| g^\varepsilon \right\|_{L^2(H^1)}^2
 \lesssim \varepsilon \| u^1 \|_{L^2_{[0,T]}(H^2)}^2
+ \varepsilon^{\frac{3}{4}} \left\| v \right\|_{L^2(H^4_x(H^{2,3}_z))}^2.
\end{equation}
Applying Lemma~\ref{lemma.vanishing} with $p = 4$, $n = 4$ and $m = 3$ 
concludes the proof of~\eqref{eq.estimate.g}.

\bigskip \noindent \textbf{Estimates for $f^\varepsilon$.} For $t \geq 0$, 
we estimate the 36 terms involved in the definition of $f^\varepsilon$ 
in~\eqref{eq.def.f}. The conclusion is that~\eqref{eq.estimate.f} holds as soon
as $v$ is bounded in $L^1(H^4_x(H^{3,4}_z))$. This can be obtained from 
Lemma~\ref{lemma.vanishing} with $p = 4$, $n = 6$ and $m = 4$. Let us give a 
few examples of some of the terms requiring the most regularity. The key point
\fm{is that} all terms of~\eqref{eq.def.f} appearing at order 
$\mathcal{O}(1)$ involve a boundary layer term and thus benefit from the
fast variable scaling gain of $\varepsilon^{\frac{1}{4}}$ in $L^2$ 
of~\cite[Lemma 3, page 150]{MR2754340}. For example, 
with~\eqref{eq.estimate.w.general}:
\begin{equation}
 \left| \eval{\partial_{zz} w}(t) \right|_{L^2}
 \lesssim \varepsilon^\frac{1}{4} \left| w(t) \right|_{H^1_x(H^{2,0}_z)}
 \lesssim \varepsilon^\frac{1}{4} \left| v(t) \right|_{H^2_x(H^{3,2}_z)}.
\end{equation}
Using~\eqref{eq.estimate.theta.h3} and~\eqref{eq.estimate.w.general}, we obtain:
\begin{equation}
 \varepsilon \left| \Delta \nabla \theta^\varepsilon(t) \right|_{L^2}
 \lesssim \varepsilon^\frac{3}{4} \left| w(t) \right|_{H^3_x(H^{1,0}_z)}
 + \left| v(t) \right|_{H^2_x(H^{0,1}_z)}
 \lesssim \varepsilon^\frac{3}{4} \left| v(t) \right|_{H^4_x(H^{2,2}_z)}.
\end{equation}
The time derivative $\eval{\partial_t w}$ can be estimated easily because the
time derivative commutes with the definition of $w$ through 
formula~\eqref{eq.def.w}. Moreover, $\partial_t v$ can be recovered from its
evolution equation~\eqref{eq.v}:
\begin{equation}
 \left| \eval{\partial_t w}(t) \right|_{L^2}
 \lesssim \varepsilon^\frac{1}{4} \left| \partial_t w(t) 
 \right|_{H^1_x(H^{0,0}_z)}
 \lesssim \varepsilon^\frac{1}{4} \left| \partial_t v(t) 
 \right|_{H^2_x(H^{1,2}_z)}
 \lesssim \varepsilon^\frac{1}{4} 
 \left( \left| v(t) \right|_{H^3_x(H^{2,4}_z)}
 +\left| \force^v(t) \right|_{H^3_x(H^{2,4}_z)} \right).
\end{equation}
The forcing term $\force^v$ is smooth and supported in $[0,T]$. As a last 
example, consider the term $(\nabla\theta^\varepsilon \cdot n) \partial_z v$.
We use the injection $H^1 \hookrightarrow L^4$ which is valid in \fm{2D and in 
3D} and estimate~\eqref{eq.estimate.theta.h2}:
\begin{equation}
 \left| (\nabla\theta^\varepsilon \cdot n) \eval{\partial_z v}(t) \right|_{L^2}
 \lesssim \left| \nabla\theta^\varepsilon(t) \right|_{H^1} 
 \left| \eval{\partial_z v}(t) \right|_{H^1}
 \lesssim \varepsilon^{\frac{1}{4}} \left| v(t) \right|_{H^3_x(H^{1,2}_z)} 
 \left| v(t) \right|_{H^2_x(H^{1,0}_z)}\fm{.}
\end{equation}
As~\eqref{eq.v.estimate} both yields $L^\infty$ and $L^1$ estimates in time, 
this estimation is enough to conclude. All remaining nonlinear convective
terms can be handled in the same way or even more easily. \fm{The pressure
term is estimated using~\eqref{eq.estimate.qpressure}.}

\bigskip

These estimates conclude the proof of \fm{small-time} global approximate null 
controllability in the general case. Indeed, both the boundary layer profile
(thanks to Lemma~\ref{lemma.vanishing}) and the remainder are small at the 
final time. \fm{Thus, as announced in Remark~\ref{remark.strong}, we have not 
only proved that there exists a weak trajectory going approximately to zero, 
but that any weak trajectory corresponding to our source terms 
$\force^\varepsilon$ and $\sigma^\varepsilon$ goes approximately to zero.}
We combine this result with the local and regularization arguments 
explained in paragraph~\ref{par.proof} to conclude the proof of 
Theorem~\ref{thm.weak.null} in the general case.

% ==============================================================================
\section{Global controllability to the trajectories}
\label{section.trajectories}
% ==============================================================================

In this section, we explain how our method can be adapted to prove 
\fm{small-time} 
global exact controllability to other states than the null equilibrium state.
Since the Navier-Stokes equation exhibits smoothing properties, all conceivable
target states must be smooth enough. Generally speaking, the exact description 
of the set of reachable states for a given controlled system is a difficult 
question. Already for the heat equation on a line segment, the complete 
description of this set is still open (see~\cite{2016arXiv160902692D} 
and~\cite{MR3551775} for recent developments on this topic). The usual 
circumvention is to study the notion of global exact controllability 
\emph{to the trajectories}. That is, we are interested in whether all known 
final states of the system are reachable from any other arbitrary initial state 
using a control:

\begin{theorem} \label{thm.trajectories}
 Let $T > 0$. \fm{Assume that the intersection of $\Gamma$ with each connected
 component of $\partial\Omega$ is smooth.} 
 Let $\bar{u} \in \mathcal{C}_w^0([0,T];\Hspace) 
 \cap L^2((0,T); H^1(\Omega))$ be a fixed weak trajectory
 of~\eqref{eq.main} \fm{with smooth~$\force$}. Let $\uinit \in 
 \Hspace$ be another initial data 
 unrelated with~$\bar{u}$. Then there exists 
 $u \in \mathcal{C}_w^0([0,T];\Hspace)
 \cap L^2((0,T); H^1(\Omega))$ a weak trajectory of~\eqref{eq.main}
 with $u(0,\cdot) = \uinit$ satisfying $ u(T,\cdot) = \bar{u}(T,\cdot)$.
\end{theorem}

The strategy is very similar to the one described in the previous sections to 
prove the global null controllability. We start with the following lemma, 
asserting \fm{small-time} global approximate controllability to smooth 
trajectories in the extended domain.

\begin{lemma} \label{lemma.thm.traj}
 Let $T > 0$. Let $(\bar{u},\bar{\force},\bar{\sigma}) 
 \in \mathcal{C}^\infty([0,T]\times\bar{\Omext})$ be a fixed 
 smooth trajectory of~\eqref{eq.main.ext}. 
 Let $\uinit \in \ldiv$ be another initial data 
 unrelated with~$\bar{u}$. For any $\delta > 0$, there exists 
 $u \in \mathcal{C}_w^0([0,T];\ldiv)
 \cap L^2((0,T); H^1(\Omext))$ a weak Leray solution of~\eqref{eq.main.ext}
 with $u(0,\cdot) = \uinit$ satisfying 
 $\left|u(T) - \bar{u}(T)\right|_{L^2(\Omext)} \leq \delta$.
\end{lemma}

\begin{proof}
We build a sequence $\use$ to achieve global approximate controllability to the 
trajectories. Still using the same scaling, we define it as:
\begin{equation}
 \use(t,x) := \frac{1}{\varepsilon} \ue \left(\frac{t}{\varepsilon},x\right),
\end{equation}
where $\ue$ solves the vanishing viscosity Navier-Stokes equation~\eqref{eq.nse}
with initial data $\varepsilon \uinit$ on the time interval $[0,T/\varepsilon]$.
As previously, this time interval will be used in two different stages. First,
a short stage of fixed length $T$ to achieve controllability of the Euler system 
by means of a return-method strategy. Then, a long stage $[T,T/\varepsilon]$, 
during which the boundary layer dissipates thanks to the careful choice of the 
boundary controls during the first stage. During the first stage, we use the 
expansion:
\begin{equation} \label{eq.exp.before}
 \ue = u^0 + \sqrt{\varepsilon} \eval{v} + \varepsilon u^{1,\varepsilon} + \ldots,
\end{equation}
where $u^{1,\varepsilon}$ is built such that 
$u^{1,\varepsilon}(0,\cdot) = \uinit$ and 
$u^{1,\varepsilon}(T,\cdot) = \bar{u}(\varepsilon T,\cdot)$. This is the main 
difference with respect to the null controllability strategy. Here, we need to 
aim for a non zero state at the first order. Of course, this is also possible 
because the state $u^{1,\varepsilon}$ is mostly transported by $u^0$ (which is 
such that the linearized Euler system is controllable). The profile 
$u^{1,\varepsilon}$ now depends on $\varepsilon$. However, since the reference 
trajectory belongs to $\mathcal{C}^\infty$, all required estimates can be made 
independent on $\varepsilon$. During this first stage, $u^{1,\varepsilon}$ 
solves the usual first-order system~\eqref{eq.u1.ext}. For large times 
$t \geq T$, we change our expansion into:
\begin{equation} \label{eq.exp.after}
 \ue = \sqrt{\varepsilon} \eval{v} 
 + \varepsilon \bar{u}(\varepsilon t,\cdot) + \ldots,
\end{equation}
where the boundary layer profile solves the homogeneous heat 
system~\eqref{eq.v.after} and $\bar{u}$ is the reference trajectory solving 
the true Navier-Stokes equation. As we have done in the case of null 
controllability, we can derive the equations satisfied by the remainders in the 
previous equations and carry on both well-posedness and smallness estimates 
using the same arguments. Changing expansion~\eqref{eq.exp.before}
into~\eqref{eq.exp.after} allows to get rid of some unwanted terms in the 
equation satisfied by the remainder. Indeed, terms such as 
$\varepsilon \Delta u^1$ or $\varepsilon (u^1\nabla) u^1$ don't appear anymore
because they are already taken into account by $\bar{u}$. One important remark
is that it is necessary to aim for $\bar{u}(\varepsilon T) \approx \bar{u}(0)$
at the linear order and not towards the desired end state $\bar{u}(T)$. Indeed, 
the inviscid stage is very short and the state will continue evolving while the 
boundary layer dissipates. This explains our choice of pivot state. We obtain:
\begin{equation}
 \left| \use(T) - \bar{u}(T) \right|_{L^2(\Omext)}
 = \mathcal{O}\left(\varepsilon^{\frac{1}{8}}\right),
\end{equation}
which concludes the proof of approximate controllability.
\end{proof}

We will also need the following regularization lemma:

\begin{lemma} \label{lemma.regularization.traj}
 Let $T > 0$. Let $\bar{u} \in \mathcal{C}^\infty([0,T] \times 
 \bar{\Omext})$ be a fixed smooth \fm{function with $\bar{u}\cdot n = 0$
     on $\BorderExt$}. 
 There exists a smooth function $C$, with $C(0) = 0$, such that,
 for any $r_* \in \ldiv$ and any $r \in \mathcal{C}_w^0([0,T];\ldiv) 
 \cap L^2((0,T); H^1(\Omext))$, weak Leray solution to:
 \begin{equation} \label{eq.ns.r.traj}
  \left\{
   \begin{aligned}
    \partial_t r - \Delta r + (\bar{u} \fm{\cdot} \nabla) r + (r 
    \fm{\cdot}\nabla) \bar{u}
    + (r \fm{\cdot} \nabla) r + \nabla \pi & = 0 && \quad \textrm{\fm{in} } 
    [0,T] \times \Omext, 
    \\
    \diverg r & = 0 && \quad \textrm{\fm{in} } [0,T] \times \Omext, \\
    r \cdot n & = 0 && \quad \textrm{\fm{on} } [0,T] \times \BorderExt, \\
    N(r) & = 0 && \quad \textrm{\fm{on} } [0,T] \times \BorderExt, \\
    r(0,\cdot) & = r_* && \quad \textrm{\fm{in} }\Omext,
   \end{aligned}
  \right.
 \end{equation}
 the following property holds true:
 \begin{equation}
  \exists t_r \in [0,T], \quad 
  \left|r(t_r,\cdot)\right|_{H^3(\Omext)}
  \leq C\left(\left|r_*\right|_{L^2(\Omext)}\right).
 \end{equation}
\end{lemma}

\begin{proof}
 This regularization lemma is \fm{easy} in our context because we assumed a lot 
 of smoothness on the reference trajectory~$\bar{u}$ and we are not demanding 
 anything on the time~$t_r$ at which the solution becomes smoother. We only 
 sketch out the steps that we go through. We repeatedly use the Korn
 inequality from~\cite[Theorem 10.2, page 299]{MR1742312} to derive estimates
 from the symmetrical part of gradients. Let $\mathbbm{P}$ denote the usual 
 orthogonal Leray projector on divergence-free vectors fields tangent to the 
 boundaries. We will use the fact $\left| \Delta r \right|_{L^2}
 \lesssim \left| \mathbbm{P} \Delta r\right|_{L^2}$ which follows from maximal
 regularity result for the Stokes problem with $\diverg r = 0$ \fm{in 
 $\Omext$}, $r\cdot n = 0$
 and $N(r) = 0$ \fm{on $\BorderExt$}. Our scheme is inspired 
 from~\cite{MR1798753}.

 \bigskip \noindent \textbf{Weak solution energy estimate}. We start with 
 the usual weak solution energy estimate \fm{(which is included in the
     definition of a weak Leray solution to~\eqref{eq.ns.r.traj})}, formally 
 multiplying~\eqref{eq.ns.r.traj} by~$r$ and integrating by parts. We obtain:
 \begin{equation} \label{eq.lemma.traj.1}
  \exists C_1, \fm{\mathrm{for~a.e.~}} t \in [0,T], \quad 
  \left| r(t) \right|_{L^2(\Omext)}^2 +
  \int_0^t \left| r(t') \right|_{H^1(\Omext)}^2 \dt'
  \leq C_1 \left| r_* \right|_{L^2(\Omext)}^2.
 \end{equation}
 In particular~\eqref{eq.lemma.traj.1} yields the existence of $0 \leq t_1 \leq 
 T/3$ such that:
 \begin{equation}
  \left| r(t_1) \right|_{H^1(\Omext)} 
  \leq \sqrt{\frac{3 C_1}{T}}
  \left| r_* \right|_{L^2(\Omext)}.
 \end{equation}
 
 \bigskip \noindent \textbf{Strong solution energy estimate}. We move on to the
 usual strong solution energy estimate, multiplying~\eqref{eq.ns.r.traj} 
 by~$\mathbbm{P} \Delta r$ and integrating by parts. We obtain:
 \begin{equation} \label{eq.lemma.traj.2}
  \exists C_2, \forall t \in [t_1,t_1 + \tau_1], \quad 
  \left| r(t) \right|_{H^1(\Omext)}^2 +
  \int_{t_1}^t \left| r(t') \right|_{H^2(\Omext)}^2 \dt'
  \leq C_2 \left| r(t_1) \right|_{H^1(\Omext)}^2,
 \end{equation}
 where $\tau_1 \leq T/3$ is a short existence time coming from the estimation 
 of the nonlinear term and bounded below as a function of
 $\left| r(t_1) \right|_{H^1(\Omext)}$. See~\cite[Theorem 6.1]{MR1798753}
 for a detailed proof. Our situation introduces an unwanted boundary term during 
 the integration by parts of 
 $\langle \partial_t r, \mathbbm{P} \Delta r \rangle$:
 \begin{equation}
  \int_{t_1}^t \int_{\BorderExt} \tanpart{D(r)n} \tanpart{\partial_t r}
  = - \int_{t_1}^t \int_{\BorderExt} (M r) \cdot \partial_t r.
 \end{equation}
 Luckily, the Navier boundary conditions helps us win one space derivative.
 When $M$ is a scalar (or a symmetric matrix), this term can be seen as a time
 derivative. In the general case, we have to conduct a parallel estimate for 
 $\partial_t r \in L^2$ by multiplying equation~\eqref{eq.ns.r.traj} by 
 $\partial_t r$, which allows us to maintain the 
 conclusion~\eqref{eq.lemma.traj.2}. In particular, this yields the existence
 of $0 \leq t_2 \leq 2T/3$ such that:
 \begin{equation}
  \left| r(t_2) \right|_{H^2(\Omext)} 
  \leq \sqrt{\frac{C_2}{\tau_1}}
  \left| r(t_1) \right|_{H^1(\Omext)}.
 \end{equation}
 
 \bigskip \noindent \textbf{Third energy estimate}. We iterate once more.
 We differentiate~\eqref{eq.ns.r.traj} with respect to time to obtain an 
 evolution equation on~$\partial_t r$ which we multiply by $\partial_t r$
 and integrate by parts. We obtain:
 \begin{equation} \label{eq.lemma.traj.3}
  \exists C_3, \forall t \in [t_2,t_2 + \tau_2], \quad 
  \left| \partial_t r(t) \right|_{L^2(\Omext)}^2 +
  \int_{t_2}^t \left| \partial_t r(t') \right|_{H^1(\Omext)}^2 \dt'
  \leq C_3 \left| \partial_t r(t_2) \right|_{L^2(\Omext)}^2,
 \end{equation}
 where $\tau_2$ is a short existence time bounded from below as a function 
 of $\left| \partial_t r(t_2) \right|_{L^2(\Omext)}$, which is bounded
 at time $t_2$ since we can compute it from equation~\eqref{eq.ns.r.traj}.
 Using~\eqref{eq.lemma.traj.3}, we deduce an $L^\infty(H^2)$ bound on $r$
 seeing~\eqref{eq.ns.r.traj} as a Stokes problem for $r$. Using the same
 argument as above, we find a time $t_3$ such that $r \in H^3$ with a
 quantitative estimate.
\end{proof}

Now we can prove Theorem~\ref{thm.trajectories}. Even though $\bar{u}$ is only
a weak trajectory on $[0,T]$, there exists $0 \leq T_1 < T_2 \leq T$ such that
$\bar{u}$ is smooth on $[T_1,T_2]$. This is a classical statement
(see~\cite[Remark 3.2]{MR645638} for the case of Dirichlet boundary conditions).
We will start our control strategy by doing nothing on $[0,T_1]$. Thus, the
weak trajectory $u$ will move from $\uinit$ to some state 
$u(T_1)$ which we will use as a new initial data. Then, we use our control to
drive $u(T_1)$ to $\bar{u}(T_2)$ at time $T_2$. After $T_2$, we choose null
controls. The trajectory $u$ follows $\bar{u}$. Hence, without loss of generality,
we can assume that $T_1 = 0$ and $T_2 = T$. This allows to work with a smooth
reference trajectory.

To finish the control strategy, we use the local result from~\cite{MR2224824}.
According to this result, there exists $\delta_{T/3} > 0$ such that, if we 
succeed to prove that there exists $0 < \tau < 2T/3$ such that 
$\left| u(\tau) - \bar{u}(\tau) \right|_{H^3(\Omext)} \leq \delta_{T/3}$,
then there exist controls driving $u$ to $\bar{u}(T)$ at time $T$. If we choose
null controls $r := u - \bar{u}$ satisfies the hypothesis 
of Lemma~\ref{lemma.regularization.traj}. Hence, there exists $\delta > 0$
such that $C(\delta) \leq \delta_{T/3}$ and we only need to build a trajectory
such that $\left| u(T/3) - \bar{u}(T/3) \right|_{\fm{L^2}(\Omext)} \leq \delta$,
which is precisely what has been proved in Lemma~\ref{lemma.thm.traj}.
This concludes the proof of Theorem~\ref{thm.trajectories}.

% ==============================================================================
\section*{Perspectives}
% ==============================================================================

The results obtained in this work can probably be extended in
following directions:
\begin{itemize}
	\item As stated in Remark~\ref{remark.strong}, \fm{for the 3D case,} it 
	would be interesting to prove that the constructed trajectory is a strong 
	solution \fm{of} the Navier-Stokes system (provided that the initial data 
	is smooth enough).	Since the first order profiles are smooth, the key 
	point is whether we can obtain strong energy estimates for the remainder 
	despite the presence of a boundary layer. \fm{In the uncontrolled setting,
    an alternative approach to the asymptotic expansion of~\cite{MR2754340}
    consists in introducing \emph{conormal Sobolev spaces} to perform energy
    estimates (see~\cite{MR2885569}).}
	\item As proposed in~\cite{MR2579376},~\cite{MR3022084} 
	then~\cite{2016arXiv160203045G}, respectively for the case of perfect 
	fluids (Euler equation) then very viscous fluids (stationary Stokes 
	equation), the notion of \emph{Lagrangian controllability} is interesting 
	for applications. It is likely that the proofs of these references can be 
	adapted to the case	of the Navier-Stokes equation with Navier boundary 
	conditions thanks to our method, since the boundary layers are located in a 
	small neighborhood of the boundaries of the domain which can be kept 
	separated from the Lagrangian trajectories of the considered movements.
    \fm{This adaptation might involve stronger estimates on the remainder.}
	\item As stated after Lemma~\ref{lemma.euler}, the hypothesis that 
	the control domain~$\Gamma$ intersects all connected components of the 
	boundary~$\partial\Omega$ of the domain is necessary to	obtain 
	controllability of the Euler equation. However, since we are dealing
	with the Navier-Stokes equation, it might be possible to release this 
	assumption, obtain partial results in its absence, or prove that it 
	remains necessary. This question is also linked to the possibility of 
	controlling a fluid-structure system where one tries to control the 
	position of a small solid immersed in a fluid domain \fm{by a control on a 
	part of the external border only}. Existence of weak 
	solutions for such a system is studied in~\cite{MR3272367}.
	\item At least for simple geometric settings of Open 
	Problem~(\hyperlink{openproblem}{OP}), our method might be adapted to the 
	challenging Dirichlet boundary condition. In this case, the 
	amplitude of the boundary layer is $\mathcal{O}(1)$ instead of 
	$\mathcal{O}(\sqrt{\varepsilon})$ here for the Navier condition. 
	This scaling deeply changes the equations satisfied by the boundary layer
	profile. Moreover, the new evolution equation satisfied by the remainder 
	involves a difficult term
	$\frac{1}{\sqrt{\varepsilon}} (\Rem \cdot n) \partial_z v$. 
	Well-posedness and smallness estimates for the remainder are much 
	harder and might involve analytic techniques. \fm{We refer to
    paragraph~\ref{par.prandtl} for a short overview of some of the difficulties
    to be expected.}
\end{itemize}
More generally speaking, we expect that the \textit{well-prepared dissipation} 
method can be applied to other fluid mechanics systems to obtain small-time 
global controllability results, as soon as asymptotic expansions for the 
boundary layers are known.

% ==============================================================================

\appendix

\section{Smooth controls for the linearized Euler equation}
\label{annex.euler}

In this appendix, we provide a constructive proof of Lemma~\ref{lemma.u1}. 
The main idea is to construct a force term~$\force^1$ such that 
$\nabla \times u^1(T,\cdot) = 0$ in~$\Omext$. Hence, the final time profile 
$U := u^1(T,\cdot)$ satisfies:
\begin{equation} \label{eq.divcurl}
\left\{
\begin{aligned}
\nabla\cdot U & = 0 && \quad \textrm{in } \Omext,\\
\rot U & = 0 && \quad \textrm{in } \Omext,\\
U \cdot n & = 0 && \quad \textrm{on } \BorderExt.\\
\end{aligned}
\right.
\end{equation}
For simply connected domains, 
this implies that $U = 0$ in $\Omext$. For 
multiply connected domains, the situation is more complex. 
Roughly speaking, a finite number of non vanishing solutions 
to~\eqref{eq.divcurl} must be ruled out by sending in appropriate 
vorticity circulations. For more details on this specific topic, 
we refer to the original references:~\cite{MR1380673} for 2D, 
then~\cite{MR1745685} for 3D. 
Here, we give an explicit construction of
a regular force term such that $\nabla \times u^1(T,\cdot) = 0$.
The proof is slightly different in the 2D and 3D settings, because the 
vorticity formulation of~\eqref{eq.u1.ext} is not exactly the same. 
In both cases, we need to build an appropriate partition of unity. 

\subsection{Construction of an appropriate partition of unity}

First, thanks to hypothesis~\eqref{eq.flow.out}, the continuity of the
flow~$\Phi^0$ and the compactness of~$\bar{\Omext}$, there exists~$\delta > 0$
such that:
\begin{equation}
 \forall x\in\bar{\Omext}, \exists t_x\in(0,T), 
 \quad
 \mathrm{dist} \left(\Phi^0(0,t_x,x), \bar{\Omega}\right) \geq \delta.
\end{equation}
Hence, there exists a smooth closed control 
region~$K \subset \bar{\Omext}$ such that $K \cap \bar{\Omega} = \emptyset$ and:
\begin{equation} \label{eq.flow.out.K}
\forall x\in\bar{\Omext}, \exists t_x\in(0,T), 
\quad
\Phi^0(0,t_x,x) \in K.
\end{equation}
Next, for each point $x \in \BorderExt \cap K$, we build a small square (resp. a 
cube) centered at $x$, included in $\R^d\setminus\bar{\Omega}$, such that one side 
(resp. one face) is included in $\Omext$ and
one side (resp. one face) is included in $\R^d \setminus \bar{\Omext}$. For each point
$x \in K \setminus \BorderExt$, we build a small square (resp. a cube) centered at $x$ 
and included in $\Omext\setminus\bar{\Omega}$. 
Since these two sets of open squares (resp. cubes) cover $K$, which
is compact, we can extract a finite number of squares (resp. cubes), labeled $C_{m}$
for $1 \leq m \leq M$, that cover $K$ and never 
intersect~$\bar{\Omega}$.

\begin{figure}[!ht]
    \centering
\begin{pspicture}(0,-2.0382237)(12.810851,2.0382237)
\definecolor{colour2}{rgb}{0.8,0.0,0.4}
\definecolor{colour3}{rgb}{0.9490196,0.9490196,0.9490196}
\definecolor{colour4}{rgb}{0.8,0.8,0.8}
\definecolor{colour1}{rgb}{0.0,0.0,0.6}
\definecolor{colour5}{rgb}{0.4,0.4,0.4}
\definecolor{colSmallSquares}{rgb}{0.2,0.2,0.2}
\psbezier[linecolor=colour2, linewidth=0.04, 
fillstyle=solid,fillcolor=colour3](7.5908513,2.0182238)(11.590852,2.0182238)(12.790852,0.4182238)(12.790852,-1.1817761873780086)(12.790852,-2.7817762)(10.671932,-0.6952897)(8.823284,-1.3655599)
\psbezier[linecolor=black, linewidth=0.04, 
fillstyle=solid,fillcolor=colour4](6.3983984,2.0131767)(3.1983986,2.0131767)(0.3983986,0.41317666)(0.3983986,-1.1868233571893319)(0.3983986,-2.7868233)(2.7983985,0.41317666)(5.1983986,-1.5868233)
\psbezier[linecolor=black, 
linewidth=0.04](6.3908515,2.0182238)(7.5908513,2.0182238)(7.9908514,2.0182238)(7.190851,2.0182238126219914)
\psbezier[linecolor=black, 
linewidth=0.04](5.190851,-1.5817761)(5.5908513,-1.9817762)(7.612473,-1.7655599)(8.812473,-1.3655599711617987)
\psbezier[linecolor=colour1, linewidth=0.04, linestyle=dashed, 
dash=0.17638889cm 0.10583334cm, 
fillstyle=solid,fillcolor=colour3](8.855717,-1.3547492)(7.6665273,-1.9061005)(5.5367975,1.8993049)(7.5908513,2.0182238126219914)
\rput[bl](9.686285,0.20387153){$\Omega$}
\rput[bl](4.1356792,0.7285686){$K$}
\rput[bl](7.729528,-0.3312636){\textcolor{colour1}{$\Gamma$}}
\rput{13.716166}(-0.20154776,-0.5544396){\psframe[linecolor=colour5, 
linewidth=0.01, dimen=outer](2.737518,-0.5817762)(1.6708515,-1.6484429)}
\psframe[linecolor=colour5, linewidth=0.01, dimen=outer](3.9926124,-0.20366298)(2.6059458,-1.5903296)
\psframe[linecolor=colour5, linewidth=0.01, dimen=outer](1.2801596,-1.2746063)(0.61349297,-1.9412731)
\rput{-61.61086}(1.4755716,-0.34111816){\psframe[linecolor=colour5, linewidth=0.01, dimen=outer](0.7850653,-1.0746064)(0.11839861,-1.741273)}
\rput{-77.32968}(2.499309,0.38014707){\psframe[linecolor=colour5, linewidth=0.01, dimen=outer](1.820537,-1.0383799)(1.1538703,-1.7050467)}
\rput{-13.425255}(0.35248125,0.93518054){\psframe[linecolor=colour5, linewidth=0.01, dimen=outer](4.482424,-0.69649315)(3.815757,-1.3631598)}
\rput{-24.218868}(0.93093336,1.8038056){\psframe[linecolor=colour5, linewidth=0.01, dimen=outer](5.002424,-0.934229)(4.3357573,-1.6008956)}
\rput{-32.04765}(1.6121639,2.5032544){\psframe[linecolor=colour5, linewidth=0.01, dimen=outer](5.497518,-1.2217761)(4.8308516,-1.8884429)}
\psframe[linecolor=colSmallSquares, linewidth=0.01, dimen=outer](2.737518,-0.10177619)(2.4308515,-0.40844285)
\psframe[linecolor=colSmallSquares, linewidth=0.01, dimen=outer](2.8575182,0.058223814)(2.5508513,-0.24844286)
\psframe[linecolor=colSmallSquares, linewidth=0.01, dimen=outer](3.057518,0.2582238)(2.7508514,-0.048442855)
\psframe[linecolor=colSmallSquares, linewidth=0.01, dimen=outer](3.217518,0.058223814)(2.9108515,-0.24844286)
\psframe[linecolor=colSmallSquares, linewidth=0.01, dimen=outer](2.8575182,0.4182238)(2.5508513,0.11155715)
\psframe[linecolor=colSmallSquares, linewidth=0.01, dimen=outer](2.737518,0.6182238)(2.4308515,0.31155714)
\psframe[linecolor=colSmallSquares, linewidth=0.01, dimen=outer](5.457518,-1.0617762)(5.1508512,-1.3684429)
\psframe[linecolor=colSmallSquares, linewidth=0.01, dimen=outer](5.2575183,-0.9417762)(4.9508514,-1.2484429)
\psframe[linecolor=colSmallSquares, linewidth=0.01, dimen=outer](5.497518,-0.8617762)(5.190851,-1.1684428)
\psframe[linecolor=colSmallSquares, linewidth=0.01, dimen=outer](5.377518,-0.62177616)(5.0708513,-0.92844284)
\psdots[linecolor=black, dotsize=0.2](3.3258677,-0.9071893)
\rput[bl](11.232231,1.5768445){$\partial\Omega \setminus \Gamma$}
\rput[bl](2.9930375,-2.0382237){$C_{m}$}
\rput[bl](1.2322308,1.5768445){$\partial\mathcal{O} \cap K$}
\end{pspicture}
    \caption{Paving the control region $K$ with appropriate squares.}
\end{figure}
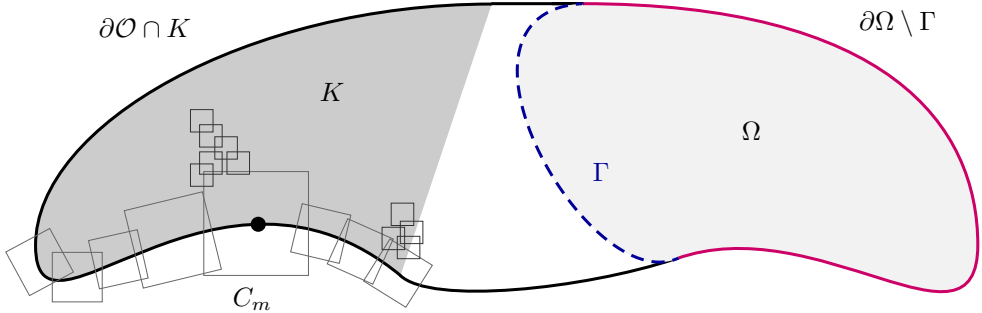

Thanks to~\eqref{eq.flow.out.K} and to the continuity of the flow~$\Phi^0$:
\begin{equation}
 \begin{split}
  \forall x\in\bar{\Omext}, \exists \epsilon_x > 0, &
  \exists t_x\in(\epsilon_x,T-\epsilon_x), \exists m_x \in \{1, \ldots M\}, 
  \forall t' \in (0,T), \forall x' \in \bar{\Omext}, \\
  & |t'-t_x|< \epsilon_x \textrm{ and } |x-x'| < \epsilon_x \Rightarrow \Phi^0\left(0,t',x'\right) \in C_{m_x}.
 \end{split}
\end{equation}
By compactness of~$\bar{\Omext}$, we can find $\epsilon > 0$ and 
balls $B_l$ for $1 \leq l \leq M$, covering~$\bar{\Omext}$, such that: 
\begin{equation} \label{eq.flow.out.cml}
\forall l \in \{1, \ldots L\}, 
\exists t_l \in (\epsilon,T-\epsilon),
\exists m_l \in \{1, \ldots M\}, 
\forall t \in (t_l - \epsilon, t_l + \epsilon),
\quad \Phi^0\left(0,t,B_l\right) \in C_{m_l}.
\end{equation}
Hence, each ball spends a positive amount of time within a given square 
(resp. cube) where we can use a local control to act on the $u^1$ profile.
This square (resp. cube) can be of one of two types as constructed above: 
either of \emph{inner} type, or of \emph{boundary} type. We also introduce
a smooth partition of unity $\eta_l$ for $1 \leq l \leq L$, such that
$0 \leq \eta_l(x) \leq 1$, $\sum \eta_l \equiv 1$ and each $\eta_l$
is compactly supported in $B_l$. Last, we introduce a smooth function 
$\beta : \R \rightarrow [0,1]$ such that $\beta \equiv 1$ 
on $(-\infty,\epsilon)$ and $\beta \equiv 0$ on $(\epsilon, +\infty)$.

\subsection{Planar case}

We consider the initial data 
$\uinit \in H^3(\Omext) \cap L^2_\mathrm{div}(\Omext)$ and we split it using the
constructed partition of unity.
Writing~\eqref{eq.u1.ext} in vorticity form, $\omega^1 := \nabla \times u^1$ can be computed as 
$\sum \omega_l$ where $\omega_l$ is the solution to:
\begin{equation} \label{eq.wl}
 \left\{
 \begin{aligned}
 \partial_t \omega_l + (\diverg u^0) \omega_l 
 + \left(u^0 \cdot \nabla\right) \omega_l & = \nabla \times \force_l
 && \quad \textrm{in } (0,T) \times \bar{\Omext}, \\
 \omega_l(0,\cdot) & = \nabla \times (\eta_l \uinit)
 && \quad \textrm{in } \bar{\Omext}.
 \end{aligned}
 \right.
\end{equation}
We consider $\bar{\omega}_l$, the solution to~\eqref{eq.wl} with $\force_l = 0$.
Setting $\omega_l := \beta(t-t_l) \bar{\omega}_l$ defines a solution to~\eqref{eq.wl}, 
vanishing at time $T$, provided that we can find $\force_l$ such that
$\nabla \times \force_l = \dot{\beta} \bar{\omega}_l$. The main difficulty
is that we need $\force_l$ to be supported in $\bar{\Omext} \setminus \Omega$.
Since $\dot{\beta} \equiv 0$ outside of $(-\epsilon, \epsilon)$,
$\dot{\beta}\bar{\omega}_l$ is supported in $C_{m_l}$ thanks
to~\eqref{eq.flow.out.cml} because the support of $\omega_l$ is transported 
by~\eqref{eq.wl}. We distinguish two cases.

\paragraph{Inner balls.}

Assume that $C_{m_l}$ is an inner square. Then $B_l$ does not intersect 
$\BorderExt$. Indeed, the streamlines of $u^0$ follow the 
boundary $\BorderExt$. If there existed $x \in B_l \cap \BorderExt$, then
$\Phi^0(0,t_l,x) \in \BorderExt$ could not belong to $C_{m_l}$, which would
violate~\eqref{eq.flow.out.cml}. Hence, $B_l$ must be an inner ball.
Then, thanks to Stokes'
theorem, the average of $\omega_l(0,\cdot)$ on $B_l$ is null (since the
circulation of $\eta_l \uinit$ along its border is null). Moreover, this 
average is preserved under the evolution by~\eqref{eq.wl} with $\force_l = 0$.
Thus, the average of $\bar{\omega}_l$ is identically null.
It remains to be checked that, if $w$ is a zero-average scalar function supported
in an inner square, we can find functions $(\force_1, \force_2)$ supported in 
the same square such that $\partial_1 \force_2 - \partial_2 \force_1 = w$.
Up to translation, rescaling and rotation, we can assume that the inner square 
is $C = [0,1]^2$. We define:
\begin{align}
\label{def.lift.2d.a}
a(x_2) & := \int_0^1 w(x_1, x_2) \dx_1, \\
\label{def.lift.2d.b}
b(x_2) & := \int_0^{x_2} a(x) \dx, \\
\label{def.lift.2d.f1}
\force_1(x_1, x_2) & := - c'(x_1) b(x_2), \\
\label{def.lift.2d.f2}
\force_2(x_1, x_2) & := - c(x_1) a(x_2) + \int_0^{x_1} w(x, x_2) \dx,
\end{align}
where $c : \R \rightarrow [0,1]$ is a smooth function with $c \equiv 0$ 
on $(-\infty,1/4)$ and $c \equiv 1$ on $(3/4,+\infty)$.
Thanks to~\eqref{def.lift.2d.a}, $a$ vanishes for $x_2 \notin [0,1]$.
Thanks to~\eqref{def.lift.2d.b}, $b$ vanishes for $x_2 \leq 0$ 
(because $a(x_2) = 0$ when $x_2 \leq 0$) and for $x_2 \geq 1$ 
(because the $b(x_2) = \int_C w = 0$ for $x_2 \geq 1$). 
Thanks to~\eqref{def.lift.2d.f1} and~\eqref{def.lift.2d.f2}, 
$(\force_1,\force_2)$ vanish outside of $C$ and 
$\partial_1 \force_2 - \partial_2 \force_1 = w$. Thus, we can
build $\force_l$, supported in $C_{m_l}$ such that 
$\nabla \times \force_l = \dot{\beta} \bar{\omega}_l$.

Moreover, thanks to this explicit construction, the spatial 
regularity of $\force_l$ is at least as good as that of~$\bar{\omega}_l$,
which is the same as that of $\nabla \times (\eta_l \uinit)$.
If $\uinit \in H^3(\Omext)$, then 
$\force_l \in \mathcal{C}^1([0,T],H^1(\Omext)) \cap 
\mathcal{C}^0([0,T],H^2(\Omext))$. This remains true after 
summation with respect to $1 \leq l \leq L$ and for 
the following constructions exposed below. If the initial data
$\uinit$ was smoother, we could also build smoother controls.

\paragraph{Boundary balls.}

Assume that $C_{m_l}$ is a boundary square. Then, $B_l$ can either be
an inner ball or a boundary ball and we can no longer assume that the 
average of $\bar{\omega}_l$ is identically null. However, the same
construction also works. Up to translation, rescaling and rotation,
we can assume that the boundary square is $C = [0,1]^2$,
with the side $x_2 = 0$ inside $\Omext$ and the side $x_2 = 1$
in $\R^2 \setminus \Omext$:
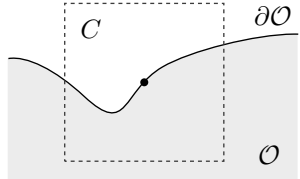
\begin{figure}[!ht]
    \centering
    \psset{xunit=.5pt,yunit=.5pt,runit=.5pt}
    \begin{pspicture}(248.03149414,141.73228455)
    {
        \newrgbcolor{curcolor}{0.9254902 0.9254902 0.9254902}
        \pscustom[linestyle=none,fillstyle=solid,fillcolor=curcolor]
        {
            \newpath
            \moveto(229.08751,113.36323909)
            \curveto(200.62696,113.17436909)(151.82094,100.87410909)(130.90001,88.55073909)
            \curveto(100.46956,70.62583909)(105.78996,35.07237909)(72.337505,66.39448909)
            \curveto(50.727455,86.62837909)(29.277698,97.15772909)(15.650008,94.92573909)
            \lineto(15.650008,2.01947909)
            \lineto(232.58751,2.01947909)
            \lineto(232.58751,113.33198909)
            \curveto(231.45753,113.36878909)(230.29456,113.37138909)(229.08751,113.36318909)
            \closepath
        }
    }
    {
        \newrgbcolor{curcolor}{0 0 0}
        \pscustom[linewidth=1,linecolor=curcolor]
        {
            \newpath
            \moveto(233.094,112.83572909)
            \curveto(231.96402,112.87252909)(230.80105,112.87502909)(229.594,112.86692909)
            \lineto(229.594,112.86697909)
            \curveto(201.13345,112.67810909)(152.32743,100.37784909)(131.4065,88.05447909)
            \curveto(100.97605,70.12957909)(106.29645,34.57607909)(72.843995,65.89822909)
            \curveto(51.233945,86.13211909)(29.784188,96.66146909)(16.156498,94.42947909)
        }
    }
    {
        \newrgbcolor{curcolor}{0.23529412 0.23529412 0.23529412}
        \pscustom[linewidth=1,linecolor=curcolor,linestyle=dashed,dash=3 3]
        {
            \newpath
            \moveto(58.29551697,136.043008)
            \lineto(177.37949371,136.043008)
            \lineto(177.37949371,16.95903888)
            \lineto(58.29551697,16.95903888)
            \closepath
        }
    }
    {
        \newrgbcolor{curcolor}{0 0 0}
        \pscustom[linestyle=none,fillstyle=solid,fillcolor=curcolor]
        {
            \newpath
            \moveto(120.83751352,76.50106053)
            \curveto(120.83751352,74.84420629)(119.49436778,73.50106055)(117.83751356,73.50106055)
            \curveto(116.18065933,73.50106055)(114.8375136,74.84420629)(114.8375136,76.50106053)
            \curveto(114.8375136,78.15791476)(116.18065933,79.5010605)(117.83751356,79.5010605)
            \curveto(119.49436778,79.5010605)(120.83751352,78.15791476)(120.83751352,76.50106053)
            \closepath
        }
    }
    \rput[bl](203,15){$\mathcal{O}$}
    \rput[bl](200,120){$\partial\mathcal{O}$}
    \rput[bl](70,110){$C$}
    \end{pspicture}
    \caption{A boundary square}
\end{figure}

We start by extending $w$ from $C \cap \bar{\Omext}$ to $C$,
choosing a regular extension operator. Then, we use the same 
formulas \eqref{def.lift.2d.a},~\eqref{def.lift.2d.b},
\eqref{def.lift.2d.f1} and~\eqref{def.lift.2d.f2}. One checks that
this defines a force which vanishes for $x_1 \leq 0$,
for $x_1 \geq 1$ and for $x_2 \leq 0$.

\subsection{Spatial case}

In 3D, each vorticity patch $\omega_l$ satisfies:
\begin{equation} \label{eq.wl.3d}
\left\{
\begin{aligned}
\partial_t \omega_l + \nabla \times (\omega_l \times u^0) & = \nabla \times \force_l
&& \quad \textrm{in } (0,T) \times \bar{\Omext}, \\
\omega_l(0,\cdot) & = \nabla \times (\eta_l \uinit)
&& \quad \textrm{in } \bar{\Omext}.
\end{aligned}
\right.
\end{equation}
Equation~\eqref{eq.wl.3d} preserves the divergence-free condition
of its initial data. Hence, proceeding as above, the only thing that 
we need to check is that, given a vector field 
$w = (w_1, w_2, w_3) : \R^3 \rightarrow \R^3$ such that:
\begin{align} 
 \label{jmc.support-g}
 \mathrm{support}(w) & \subset (0,1)^3, \\
 \label{jmc.div-g=0}
 \diverg(w) & =0,
\end{align}
we can find a vector field $\force = (\force_1, \force_2, \force_3)
: \R^3 \rightarrow \R^3$ such that:
\begin{align}
\label{jmc.eq-1}
\partial_2 \force_3 - \partial_3 \force_2 &= w_1,
\\
\label{jmc.eq-2}
\partial_3 \force_1-\partial_1 \force_3 &=w_2,
\\
\label{jmc.eq-3}
\partial_1 \force_2-\partial_2 \force_1&=w_3,
\\
\label{jmc.support-f}
\text{support} (\force) &\subset (0,1)^3.
\end{align}
As in the planar case, we distinguish the case of inner and boundary cubes.

\paragraph{Inner cubes.}

Let $a\in C^\infty(\R,\R)$ be such that:
\begin{align}
\label{jmc.int-a=1}
\int_{0}^{1} a(x)\dx&=1,
\\
\label{jmc.support-a}
\text{support} (a)&\subset (0,1).
\end{align}
We define:
\begin{align}
\label{jmc.def-f1}
\force_1(x_1,x_2,x_3)&:=a(x_1)h(x_2,x_3) \\
\label{jmc.def-f2}
\force_2(x_1,x_2,x_3)&:= \int_{0}^{x_1} (\partial_2 \force_1 + w_3)(x,x_2,x_3)\dx, \\
\label{jmc.def-f3}
\force_3(x_1,x_2,x_3)&:= \int_{0}^{x_1} (\partial_3 \force_1 -w_2)(x,x_2,x_3)\dx,
\end{align}
where $h:\R^2\to \R$ will specified later on.
From \eqref{jmc.def-f2}, one has \eqref{jmc.eq-3}. 
From \eqref{jmc.def-f3}, one has \eqref{jmc.eq-2}. 
From \eqref{jmc.support-g}, \eqref{jmc.div-g=0}, \eqref{jmc.def-f2} 
\eqref{jmc.def-f2} and \eqref{jmc.def-f3}, one has \eqref{jmc.eq-1}. 
Using \eqref{jmc.support-g}, \eqref{jmc.def-f1}, \eqref{jmc.def-f2} and
\eqref{jmc.def-f3} one checks that \eqref{jmc.support-f} holds if $h$ satisfies
\begin{align}
\label{jmc.support-h}
\text{support}(h) &\subset (0,1)^2, 
\\
\label{jmc.h2=}
\partial_2 h (x_2,x_3)&=W_2(x_2,x_3),
\\
\label{jmc.h3=}
\partial_3 h(x_2,x_3)&=W_3(x_2,x_3),
\end{align}
where
\begin{align}
\label{jmc.def-G2}
W_2(x_2,x_3):= -\int_{0}^{1} w_3(x,x_2,x_3) \dx,\\
\label{jmc.def-G3}
W_3(x_2,x_3):=\int_{0}^{1} w_2(x,x_2,x_3) \dx.
\end{align}
From \eqref{jmc.support-g}, \eqref{jmc.div-g=0}, \eqref{jmc.def-G2} and \eqref{jmc.def-G3}, one has:
\begin{gather}\label{jmc.support-G}
\text{support}(W_2) \subset (0,1)^2, \, 
\text{support}(W_3)\subset (0,1)^2, \\
\label{jmc.curl-G=0}
\partial_2 W_3-\partial_3 W_2=0.
\end{gather}
We define $h$ by
\begin{equation}
\label{jmc.defh}
h(x_2,x_3):=\int_{0}^{x_2} W_2(x,x_3) \dx,
\end{equation}
so that \eqref{jmc.h2=} holds. From \eqref{jmc.support-G}, \eqref{jmc.curl-G=0} 
and \eqref{jmc.defh}, one gets \eqref{jmc.h3=}. Finally, from \eqref{jmc.def-G2}, 
\eqref{jmc.support-G} and \eqref{jmc.defh} one sees that \eqref{jmc.support-h} 
holds if and only if:
\begin{equation}\label{jmc.property-k}
k(x_3)=0,
\end{equation}
where
\begin{equation}\label{jmc.def-k}
k(x_3):=\int_{0}^{1}\int_{0}^{1} w_3(x_1,x_2,x_3)\dx_1 \dx_2.
\end{equation}
Using \eqref{jmc.support-g}, \eqref{jmc.div-g=0} and \eqref{jmc.def-k}, 
one sees that $k' \equiv 0$ and $\text{support}(k) \subset (0,1)$,
which implies \eqref{jmc.property-k}.

\paragraph{Boundary cubes.}

Now we consider a boundary cube. Up to translation, scaling and rotation,
we assume that we are considering the cube $C = [0,1]^3$ with the face
$x_1 = 0$ lying inside $\Omext$ and the face $x_1 = 1$ lying in 
$\R^3 \setminus \Omext$. Similarly as in the planar case, we choose a 
regular extension of $w$ to $C$. We set $\force_1 = 0$ and
we define $\force_2$ by \eqref{jmc.def-f2} and 
$\force_3$ by \eqref{jmc.def-f3}. 
One has \eqref{jmc.eq-1}, \eqref{jmc.eq-2}, \eqref{jmc.eq-3}
in $C \cap \bar{\Omext}$ with $\text{support}(\force) \cap \bar{\Omext} \subset C$.

% ==============================================================================

\bibliographystyle{plain}
\bibliography{navier_slip}

\end{document}